\documentclass[12pt,reqno]{amsart} 
\usepackage[a4paper, margin=3cm]{geometry}

\usepackage{amsthm}

\usepackage{bbm}
\usepackage{mathrsfs}

\usepackage[utf8]{inputenc}
\usepackage{xypic}

\usepackage{tikz-cd}

\usepackage{enumitem}

\usepackage{float}
\usepackage{amsbsy}
\usepackage[all]{xy}\usepackage{xypic}
\usepackage{braket}
\usepackage[colorlinks = true,citecolor=black]{hyperref}

\usepackage{amsmath}

\newcommand\myeq{\mathrel{\stackrel{\makebox[0pt]{\mbox{\normalfont\tiny loc}}}{=}}}

\hypersetup{
     colorlinks=true,
     linkcolor=blue,
     filecolor=blue,
     citecolor = blue,      
     urlcolor=cyan,
     }

\newtheorem{theorem}{Theorem}[section]
\newtheorem*{theorem*}{Theorem}
\newtheorem{theorem-non}{Theorem}
\newtheorem{lemma-non}{Lemma}

\theoremstyle{definition} 
\newtheorem{thm}{Theorem}

\theoremstyle{definition} 
\newtheorem{corollarynon}{Corollary}

\newtheorem{conjecture-non}{Conjecture}

\newtheorem{corollary-non}{Corollary}
\newtheorem{proposition}[theorem]{Proposition}
\newtheorem{lemma}[theorem]{Lemma}
\newtheorem*{lemma*}{Lemma}
\newtheorem{corollary}[theorem]{Corollary}

\newtheorem*{conjecture*}{Conjecture}

\newtheorem{problem}[theorem]{Problem}
\theoremstyle{definition}
\newtheorem{definition}[theorem]{Definition}
\newtheorem{example}[theorem]{Example}

\theoremstyle{remark}
\newtheorem{remark}[theorem]{Remark}

\DeclareMathOperator{\im}{im}

\numberwithin{equation}{section}

\begin{document}
\title[Levi-Civita Ricci-flat metrics on non-K\"{a}hler Calabi-Yau manifolds]{Levi-Civita Ricci-flat metrics on non-K\"{a}hler Calabi-Yau manifolds}

\author{Eder M. Correa}


\address{{UFMG, Avenida Ant\^{o}nio Carlos, 6627, 31270-901 Belo Horizonte - MG, Brazil}}
\address{E-mail: {\rm edermc@ufmg.br.}}

\begin{abstract} 
In this paper, we provide new examples of Levi-Civita Ricci-flat Hermitian metrics on certain compact non-K\"{a}hler Calabi-Yau manifolds, including every compact Hermitian Weyl-Einstein manifold, every compact locally conformal hyperK\"{a}hler manifold, certain suspensions of Brieskorn manifolds, and every generalized Hopf manifold provided by suspensions of exotic spheres. These examples generalize previous constructions on Hopf manifolds. Additionally, we also construct new examples of compact Hermitian manifolds with nonnegative first Chern class that admit constant strictly negative Riemannian scalar curvature. Further, we remark some applications of our main results in the study of the Chern-Ricci flow on compact Hermitian Weyl-Einstein manifolds. In particular, we describe the Gromov-Hausdorff limit for certain explicit finite-time collapsing solutions which generalize previous constructions on Hopf manifolds.
\end{abstract}

\maketitle

\hypersetup{linkcolor=black}
\tableofcontents

\hypersetup{linkcolor=black}
\section{Introduction}
Given a compact Hermitian manifold $(M,g,J)$, with fundamental $2$-form $\Omega$, as shown in \cite[Theorem 1.2]{LiuYang}, the first Levi-Civita Ricci curvature $\mathfrak{R}^{(1)}(\Omega)$ of $(M,g,J)$ represents its first Aeppli-Chern class\footnote{See \cite[Definition 1.1]{LiuYang}. For more details on Aeppli cohomology, see \cite{Angella} and references therein.} $c_{1}^{AC}(M) \in H^{1,1}_{A}(M)$. More precisely, 
\begin{equation}
\label{LCRcurvarture}
\mathfrak{R}^{(1)}(\Omega) = {\rm{Ric}}^{(1)}(\Omega) - \frac{1}{2} \big ( \partial \partial^{\ast}\Omega +  \bar{\partial} \bar{\partial}^{\ast} \Omega\big),
\end{equation}
where ${\rm{Ric}}^{(1)}(\Omega) \myeq -\frac{1}{2}{\rm{d}}{\rm{d}}^{c}\log(\det(\Omega))$ is the associated (first) Chern Ricci curvature. By the celebrated Calabi-Yau theorem \cite{Yau} a compact K\"{a}hler manifold has $c_{1}(M) = 0$ if and only if it has a Ricci-flat K\"{a}hler metric, i.e. ${\rm{Ric}}(g) = 0$. In the compact non-K\"{a}hler setting, since $\mathfrak{R}^{(1)}(\Omega) = 0$ implies $c_{1}^{AC}(M) = 0$, a natural question to ask inspired by the Calabi-Yau theorem is the following:

\begin{problem}[\cite{LiuYang}]
\label{question1}
On a compact complex manifold $M$, if $c_{1}^{AC}(M) = 0$, does there exist a smooth Levi-Civita Ricci-flat Hermitian metric $\Omega$, i.e. such that $\mathfrak{R}^{(1)}(\Omega) = 0$? 
\end{problem}

From Eq. (\ref{LCRcurvarture}), the Levi-Civita Ricci-Flat condition $\mathfrak{R}^{(1)}(\Omega) = 0$ is equivalent to 
\begin{equation}
\label{LCRFEq}
{\rm{Ric}}^{(1)}(\Omega) = \frac{1}{2} \big ( \partial \partial^{\ast}\Omega +  \bar{\partial} \bar{\partial}^{\ast} \Omega\big).
\end{equation}
Since there are non-elliptic terms on the right-hand side of Eq. (\ref{LCRFEq}), it is not of Monge-Amp\`{e}re type. Thus, it is particularly challenging to solve such equations. Since  $c_{1}^{AC}(M) = 0$, if $c_{1}(M) = 0$, it is very natural to study the Problem \ref{question1} on non-K\"{a}hler Calabi-Yau manifolds, i.e. compact non-K\"{a}hler manifolds satisfying $c_{1}(M) = 0$. It is known that the Hopf manifold $M = S^{2n+1} \times S^{1}$ is non-K\"{a}hler Calabi-Yau. Moreover, since the first Bott-Chern class of $M$ does not vanish, i.e. $c_{1}^{BC}(M) \neq 0$, there does not exist a Hermitian metric on $M$, such that ${\rm{Ric}}^{(1)}(\Omega) = 0$. On the other hand, as it was shown in \cite{LiuYang}, every Hopf manifold $M = S^{2n+1} \times S^{1}$ admits a Hermitian metric $\Omega$ satisfying $\mathfrak{R}^{(1)}(\Omega) = 0$. Inspired by this example and by Problem \ref{question1}, in this paper we generalize the construction provided in \cite{LiuYang} on Hopf manifolds to a more general setting of locally conformal K\"{a}hler manifolds obtained as certain suspensions of Sasaki-Einstein manifolds. More precisely, we prove the following:
\begin{thm}
\label{T1}
Let $(Q,g_{SE})$ be a compact Sasaki-Einstein manifold, $\phi \colon Q \to Q$ a Sasaki automorphism, and $c > 0$. Then, the suspension $\Sigma_{\phi,c}(Q)$ by $(\phi,c)$ of $Q$ admits a Levi-Civita Ricci-flat Hermitian metric.
\end{thm}
Let us observe that by a suspension of $Q$ by $(\phi,c)$ we mean
\begin{equation}
\Sigma_{\phi,c}(Q) := \frac{Q \times [0,\log(c)]}{(\phi(x),0) \sim (x,\log(c))}.
\end{equation}
In particular, since $b_{1}(\Sigma_{\phi,c}(Q)) = 1$ (e.g. \cite{Verbitsky}), it follows that $\Sigma_{\phi,c}(Q)$ cannot be K\"{a}hler. From Theorem \ref{T1}, one can recover the Levi-Civita Ricci-flat Hermitian metric on Hopf manifolds provided in \cite{LiuYang} just by taking $Q = S^{2n+1}$ and $\phi = {\rm{id}}$. The main idea to prove Theorem \ref{T1} is the following: Considering the locally conformal K\"{a}hler metric $\Omega$ on $\Sigma_{\phi,c}(Q)$ induced from the Calabi-Yau metric of the metric cone $Q \times \mathbbm{R}_{+}$, e.g. \cite{Gini}, one can write explicitly the Levi-Civita Ricci-Flat condition for the perturbed Hermitian metric 
\begin{equation}
\Omega_{\zeta} := \Omega + 2\zeta\mathfrak{R}^{(1)}(\Omega), 
\end{equation}
such that $\zeta > -1$. Then, we show that one can always find $\zeta$, such that $\mathfrak{R}^{(1)}(\Omega_{\zeta}) = 0$. The key point which allows us to solve Eq. (\ref{LCRFEq}) is that $(\Sigma_{\phi,c}(Q),\Omega)$ is a compact Hermitian Weyl-Einstein manifold. This fact allows us to perform all computations needed using differential forms and to reduce the problem to a simple equation in terms of $\zeta$ just like in the case Hopf manifolds (cf. \cite[Theorem 6.2]{LiuYang}). The framework on Sasaki-Einstein geometry of Theorem \ref{T1} allows us to obtain several examples of Levi-Civita Ricci-flat manifolds. In fact, since every 3-Sasakian manifold is, in particular, Sasaki-Einstein, see for instance \cite{3Einstein}, \cite{Kashiwada2}, we obtain the following result.
\begin{corollarynon}
\label{C1}
Let $(Q,g_{Q})$ be a compact $3$-Sasakian manifold, $\phi \colon Q \to Q$ a Sasaki automorphism, and $c > 0$. Then, the suspension $\Sigma_{\phi,c}(Q)$ by $(\phi,c)$ of $Q$ admits a Levi-Civita Ricci-flat Hermitian metric.
\end{corollarynon}

Also, from \cite{Verbitsky} and Theorem \ref{T1}, we have a positive answer for the question proposed in Problem \ref{question1} in the case that $M$ is a compact Hermitian Weyl-Einstein manifold. In fact, we obtain the following.

\begin{corollarynon}
\label{C2}
Every compact Hermitian Weyl-Einstein manifold admits a Levi-Civita Ricci-flat Hermitian metric. In particular, every compact locally conformal hyperK\"{a}hler manifold admits a Levi-Civita Ricci-flat Hermitian metric.
\end{corollarynon}

Our next result is an application of Theorem \ref{T1} in the setting of quasi-regular Sasaki-Einstein manifolds provided by links of hypersurfaces singularities. More precisely, from the result provided in \cite[Theorem 1.4]{LiuSanoTasin}, see also \cite[Conjecture 4]{Einsteinsphere}, and Theorem \ref{T1}, we have the following corollary.

\begin{corollarynon}
\label{C3}
Let ${\rm{L}}({\bf{a}}) := Y({\bf{a}}) \cap S^{2n+1}$ be the link of a Brieskorn-Pham singularity
\begin{equation}
Y({\bf{a}}) := \Big (  z_{0}^{a_{0}} + \cdots + z_{n}^{a_{n}} = 0 \Big) \subset \mathbbm{C}^{n+1},
\end{equation}
such that $n \geq 3$. Assume that $a_{0} \leq \cdots \leq a_{n}$. Then ${\rm{L}}({\bf{a}}) \times S^{1}$ admits a Levi-Civita Ricci-flat Hermitian metric if 
\begin{equation}
1 < \sum_{j = 0}^{n}\frac{1}{a_{j}} < 1 + \frac{n}{a_{n}}.
\end{equation}
\end{corollarynon}

In particular, from  Theorem \ref{T1} and \cite{LiuSanoTasin}, one can obtain Levi-Civita Ricci-flat Hermitian metrics on generalized Hopf manifolds \cite{GHM}. In fact, we obtain the following corollary which generalizes \cite[Theorem 1.5]{LiuSanoTasin}.

\begin{corollarynon}
\label{C4}
Let ${\bf{\Sigma}}$ be an odd dimensional homotopy sphere which bounds a parallelizable manifold. Then ${\bf{\Sigma}} \times S^{1}$ admits a Levi-Civita Ricci-flat Hermitian metric.
\end{corollarynon}

Given two homotopy spheres ${\bf{\Sigma}}$ and ${\bf{\Sigma}}'$ of dimension $\geq 5$, then we have that ${\bf{\Sigma}} \times S^{1}$ is diffeomorphic to ${\bf{\Sigma}}' \times S^{1}$ if and only if ${\bf{\Sigma}}$ is diffeomorphic to ${\bf{\Sigma}}'$, see for instance \cite[Proposition 3]{GHM}. Therefore, in the setting of Corollary \ref{C4} the underlying differentiable manifolds ${\bf{\Sigma}} \times S^{1}$ are exotic. Additionally, inspired by \cite[Theorem 6.4]{LiuSanoTasin}, we prove the following result.
\begin{thm}
\label{T2}
Let $(Q,g_{SE})$ be a compact Sasaki-Einstein manifold, $\phi \colon Q \to Q$ a Sasaki automorphism, and $c > 0$. Then, the suspension $\Sigma_{\phi,c}(Q)$ by $(\phi,c)$ of $Q$ admits three different Hermitian metrics $\Omega_{i}$, $i = 1,2,3$, satisfying the following properties:
\begin{enumerate}
\item[{\rm{(1)}}] ${\rm{Ric}}^{(1)}(\Omega_{1}) = {\rm{Ric}}^{(1)}(\Omega_{2}) = {\rm{Ric}}^{(1)}(\Omega_{3}) \geq 0$;
\item[{\rm{(2)}}] $\Omega_{1}$ has {\bf{strictly positive}} Riemannian scalar curvature;
\item[{\rm{(3)}}] $\Omega_{2}$ has {\bf{zero}} Riemannian scalar curvature;
\item[{\rm{(4)}}] $\Omega_{3}$ has {\bf{strictly negative}} Riemannian scalar curvature.
\end{enumerate}
In particular, all compact Hermitian manifolds of the previous corollaries admit three different Hermitian metrics satisfying the above properties.
\end{thm}
As it can be seen, Theorem \ref{T2} generalizes some ideas introduced in \cite{LiuSanoTasin}. In particular, it provides a huge class of examples of compact Hermitian manifolds with nonnegative first Chern class which admit Hermitian metrics with strictly negative Riemannian scalar curvature. This fact contrasts with the setting of compact K\"{a}hler manifolds with nonnegative first Chern class. Further, we also make some remarks related to applications of the ideas used in the proof of Theorem \ref{T1} and Theorem \ref{T2} in order to construct explicit solutions of the Chern-Ricci flow (e.g. \cite{Gill}, \cite{TosattiWeinkove}) on compact Hermitian Weyl-Einstein manifolds. More precisely, we prove the following.
\begin{thm}
\label{T3}
Let $(M,g,J)$ be a compact Hermitian Weyl-Einstein manifold, then there exists an explicit solution $g(t)$ of the Chern-Ricci flow on $M$ for $ t \in [0,\frac{2}{n})$, starting at $g$, satisfying the following properties:
\begin{enumerate}
\item ${\rm{Vol}}(M,g(t)) \to 0$ as $t \to \frac{2}{n}$ (i.e. $g(t)$ is finite-time collapsing);
\item $\lim_{t \to \frac{2}{n}}g(t) = h_{T}$, where $h_{T}$ is a nonnegative symmetric tensor on $M$;
\item The Chern scalar curvature of $g(t)$ blows up like $(n-1)/(\frac{2}{n}- t)$;
\item ${\rm{scal}}(g(t)) \to -\infty$ as $t \to \frac{2}{n}$;
\item $\lim_{t \to \frac{2}{n}}d_{GH}\big ((M,d_{t}),(S^{1},d_{S^{1}}) \big ) = 0$,
\end{enumerate}
where $d_{t}$ is the distance induced by $g(t)$ on $M$ and $d_{S^{1}}$ is the distance on the unit circle $S^{1}$ induced by a suitable scalar multiple of the standard Riemannian metric.
\end{thm}

As it can be seen, the results of Theorem \ref{T3} generalize some previous results provided in \cite{TosattiWeinkove}, \cite{GillSmith}, and \cite{TosattiWeinkove2}, for Hopf manifolds $S^{2n+1} \times S^{1}$. The proof which we present for Theorem \ref{T3} takes into account the framework on Sasaki-Einstein geometry which underlies compact Hermitian Weyl-Einstein manifolds. Therefore, the results of Theorem \ref{T3} hold for any suspension of a Sasaki-Einstein manifold as in the previous corollaries and theorems.

\subsection*{Outline of the paper} This paper is organized as follows. In Section \ref{S2}, we present some generalities on Hermitian geometry, Sasaki geometry, and Hermitian Weyl-Einstein geometry. In Section \ref{S3}, we prove Theorem \ref{T1}, Theorem \ref{T2}, and we provide a huge class of examples which illustrate our results. In Appendix \ref{Appendix}, we prove Theorem \ref{T3} and we illustrate the result by means of an explicit example.

\section{Preliminary results}
\label{S2}
In this section, we review some basic definitions and results related to Hermitian geometry \cite{Gauduchonconnection}, \cite{Gauduchon}, \cite{Wells}, \cite{LiuYang}, Calabi-Yau cones \cite{Sparks}, \cite{BoyerGalicki}, and Hermitian Weyl-Einstein geometry \cite{Vaisman}, \cite{Dragomir}, \cite{Gauduchon}, \cite{Gauduchon2}, \cite{Verbitsky}. 

\subsection{Generalities on Hermitian manifolds} Let $(M,g,J)$ be a Hermitian manifold of complex dimension $n$. Denoting by $\Omega:=g(J\otimes \mathbbm{1})$ the associated fundamental $2$-form, consider
\begin{equation}
\label{primitive}
({\rm{d}}\Omega)_{0} := {\rm{d}}\Omega - \frac{1}{n-1}{\rm{L}}(\Lambda({\rm{d}}\Omega)),
\end{equation}
where ${\rm{L}} = \Omega \wedge (-)$ is the Lefschetz operator and $\Lambda$ is its adjoint \cite{Wells}. Since $\Lambda({\rm{d}}\Omega)$ is primitive, it follows that $\Lambda({\rm{L}}(\Lambda({\rm{d}}\Omega))) = (n-1)\Lambda({\rm{d}}\Omega)$, thus $\Lambda(({\rm{d}}\Omega)_{0}) = 0$, i.e., the primitive part of ${\rm{d}}\Omega$ is given by Eq. (\ref{primitive}). In the above setting, we have the following definition.
\begin{definition}
\label{LeeDef}
The Lee $1$-form of a Hermitian manifold $(M,\Omega,J)$ of complex dimension $n$ is defined by
\begin{equation}
\label{Lee}
\theta:= \frac{1}{n-1}\Lambda({\rm{d}}\Omega).
\end{equation}
\end{definition}
\begin{remark}
The Lee form $\theta$ associated to a Hermitian manifold plays an important role in the study of the classification of Hermitian structures, e.g. \cite{GrayHervella}.
\end{remark}
Now we consider the following result.
\begin{lemma}
 Let $(M,\Omega,J)$ be a Hermitian manifold of complex dimension $n$. Then 
\begin{equation}
\label{Lema1}
{\rm{d}}(\Omega^{n-1}) = (n-1)\theta \wedge \Omega^{n-1}.
\end{equation}
\end{lemma}
\begin{remark}
The above lemma follows directly from the fact that
\begin{equation}
\Lambda(({\rm{d}}\Omega)_{0}) = 0 \iff ({\rm{d}}\Omega)_{0}\wedge \Omega^{n-2} = 0,
\end{equation}
see for instance \cite[Corollary 3.13]{Wells}.
\end{remark}
Let $\nabla^{Ch}$ be the Chern connection of a Hermitian manifold $(M,g,J)$, and let $\mathcal{T}_{\nabla}$ be its torsion. Considering the $1$-form\footnote{Notice that ${\rm{tr}}(\mathcal{\mathcal{T}}_{\nabla}) = \Lambda(\rm{d}\Omega)$, see for instance \cite{Gauduchonconnection}, \cite{Gauduchon}.} 
\begin{equation}
{\rm{tr}}(\mathcal{\mathcal{T}}_{\nabla})(X) := {\text{trace}}\big (Y \mapsto \mathcal{\mathcal{T}}_{\nabla}(X,Y) \big ), \ \ \ \forall X \in \mathfrak{X}(M),
\end{equation}
we define the torsion $(1,0)$-form associated to $\nabla^{Ch}$ as being
\begin{equation}
\tau := {\rm{tr}}(\mathcal{T}_{\nabla})^{1,0} = \Lambda(\partial \Omega).
\end{equation}
From Definition \ref{LeeDef}, it follows that $\theta = \frac{1}{n-1}(\tau + \overline{\tau})$. By considering the Hodge $\ast$-operator defined by $g$ and the associated codifferential $\delta := - \ast {\rm{d}} \ast $, from the decomposition ${\rm{d}}= \partial + \bar{\partial}$, we have $\delta := \partial^{\ast} + \bar{\partial}^{\ast}$, such that 
\begin{equation}
\partial^{\ast} := - \ast \bar{\partial} \ast, \ \ \ \bar{\partial}^{\ast}:= - \ast  \partial \ast.
\end{equation}
In the above context, since $\ast \Omega^{k} = \frac{k!}{(n-1)!}\Omega^{n-k}$, we have from Lemma \ref{Lema1} that
\begin{center}
$\displaystyle{\delta \Omega =  -\frac{1}{(n-1)!}\ast{\rm{d}}(\Omega^{n-1}) = -\frac{1}{(n-2)!}}\ast \big (\theta \wedge \Omega^{n-1}\big) = - \frac{1}{(n-2)!}\ast {\rm{L}}^{n-1}(\theta).$
\end{center}
Considering $J(\theta) = - \theta \circ J = \frac{\sqrt{-1}}{n-1}(\overline{\tau} - \tau)$, since $\Lambda(\theta) = 0$, it follows that $\ast {\rm{L}}^{n-1}(\theta) = (n-1)!J(\theta)$, e.g. \cite[Theorem 3.16]{Wells}. Therefore, we conclude that 
\begin{equation}
\theta = \frac{1}{n-1}J(\delta \Omega).
\end{equation}
From above, we obtain the following description for the torsion $(1,0)$-form associated to the Chern connection $\nabla^{Ch}$.

\begin{lemma}
Let $(M,\Omega,J)$ be a compact Hermitian manifold and $\tau$ the torsion $(1,0)$-form of the associated to the Chern connection $\nabla^{Ch}$. Then
\begin{equation}
\tau = - \sqrt{-1} \bar{\partial}^{\ast} \Omega.
\end{equation}
\end{lemma}

\subsection{Levi-Civita Ricci curvature} Given a Hermitian manifold $(M,\Omega,J)$, consider its first Chern Ricci curvature  
\begin{equation}
{\rm{Ric}}^{(1)}(\Omega) \myeq -\frac{1}{2}{\rm{d}}{\rm{d}}^{c}\log(\det(\Omega)),
\end{equation}
such that ${\rm{d}}^{c} = J\circ {\rm{d}}$. It is well-known that ${\rm{Ric}}^{(1)}(\Omega) \in 2\pi c_{1}(M)$ and, if ${\rm{d}}\Omega = 0$, it follows that 
\begin{equation}
{\rm{Ric}}^{(1)}(\Omega) = {\rm{Ric}}(g)(J \otimes \mathbbm{1}),
\end{equation}
where ${\rm{Ric}}(g)$ is the Ricci tensor associated to Riemannian metric $g = \Omega(\mathbbm{1} \otimes J)$. In the case that ${\rm{d}}\Omega \neq 0$, we have several different types of Ricci curvatures, e.g. \cite{LiuYang}. Besides the first Chern Ricci curvature, we also shall consider in this paper the following notion of Ricci curvature.

\begin{definition}[Theorem 1.12, \cite{LiuYang}]
\label{LCRicci}
Let $(M,\Omega,J)$ be a compact Hermitian manifold. The first Levi-Civita Ricci curvature of $(M,\Omega,J)$ is defined by
\begin{equation}
\mathfrak{R}^{(1)}(\Omega) := {\rm{Ric}}^{(1)}(\Omega) - \frac{1}{2} \big ( \partial \partial^{\ast}\Omega +  \bar{\partial} \bar{\partial}^{\ast} \Omega\big).
\end{equation}
\end{definition}

Given a compact Hermitian manifold $(M,\Omega,J)$, we have the Bott-Chern cohomology and the Aeppli
cohomology of $(M,\Omega,J)$ defined, respectively, by
\begin{center}
$\displaystyle{H^{p,q}_{BC}(M) := \frac{\ker({\rm{d}}) \cap \Omega^{p,q}(M)}{\im(\partial \bar{\partial}) \cap \Omega^{p,q}(M)}, \ \ H^{p,q}_{A}(M) := \frac{\ker(\partial \bar{\partial}) \cap \Omega^{p,q}(M)}{\im(\partial) \cap \Omega^{p,q}(M) + \im(\bar{\partial}) \cap \Omega^{p,q}(M)}}$.
\end{center}
For more details, see for instance \cite{Angella} and references therein.
\begin{definition}[Definition 1.1, \cite{LiuYang}]
Let ${\bf{L}} \to M$ be a holomorphic line bundle over $M$. The first Aeppli-Chern class of ${\bf{L}}$ is defined by 
\begin{equation}
c_{1}^{AC}({\bf{L}}):=  \bigg [ \frac{\sqrt{-1}}{2\pi}{\bf{\Theta}}(\nabla) \bigg]_{A} \in H^{1,1}_{A}(M),
\end{equation}
where ${\bf{\Theta}}(\nabla)$ is the curvature of the Chern connection $\nabla = {\rm{d}} + {\rm{d}}\log({\bf{H}})$, for some Hermitian metric ${\bf{H}}$ on ${\bf{L}}$. In particular, the first Aeppli-Chern class of $M$ is defined by $c_{1}^{AC}(M) := c_{1}^{AC}({\bf{K}}_{M}^{-1})$.
\end{definition}
\begin{remark}
In the setting of the above definition, similarly, we define the first Bott-Chern class of ${\bf{L}} \to M$ by 
\begin{equation}
c_{1}^{BC}({\bf{L}}):=  \bigg [ \frac{\sqrt{-1}}{2\pi}{\bf{\Theta}}(\nabla)\bigg]_{BC} \in H^{1,1}_{BC}(M).
\end{equation}
The first Bott-Chern class of $M$ is defined by $c_{1}^{BC}(M) := c_{1}^{BC}({\bf{K}}_{M}^{-1})$.
\end{remark}
From Definition \ref{LCRicci}, given a compact Hermitian manifold $(M,\Omega,J)$, it follows that 
\begin{center}
$\displaystyle c_{1}^{AC}(M) = \bigg [\frac{{\rm{Ric}}^{(1)}(\Omega)}{2\pi} \bigg]_{A} = \bigg [ \frac{\mathfrak{R}^{(1)}(\Omega)}{2\pi}\bigg]_{A}$.
\end{center}
The vanishing of the cohomology classes described above are related by the following result.
\begin{proposition}[Corollary 1.3, \cite{LiuYang}]
Let $M$ be a complex manifold. Then
\begin{equation}
c_{1}^{BC}(M) = 0 \Longrightarrow c_{1}(M) = 0 \Longrightarrow c_{1}^{AC}(M) = 0.
\end{equation}
Moreover, on a complex manifold satisfying the $\partial \bar{\partial}$-lemma, we have
\begin{equation}
c_{1}^{BC}(M) = 0 \iff c_{1}(M) = 0 \iff c_{1}^{AC}(M) = 0.
\end{equation}
\end{proposition}

\subsection{Scalar curvatures} Given a Hermitian manifold $(M,g,J)$, we can define its Chern scalar curvature as $s_{C}(g):= {\rm{tr}}_{\Omega}({\rm{Ric}}^{(1)}(\Omega))$, such that 
\begin{equation}
\frac{{\rm{tr}}_{\Omega}({\rm{Ric}}^{(1)}(\Omega))}{n}\Omega^{n} = {\rm{Ric}}^{(1)}(\Omega) \wedge \Omega^{n-1}.
\end{equation}
By considering the underlying Riemannian structure of $(M,g,J)$, we also have the notion of Riemannian scalar curvature ${\rm{scal}}(g):= {\rm{tr}}_{g}({\rm{Ric}}(g))$. In the particular case that ${\rm{d}}\Omega = 0$, i.e. when $(M,g,J)$ is K\"{a}hler, these two notions of scalar curvature are related by
\begin{equation}
{\rm{scal}}(g) = 2s_{C}(g).
\end{equation}
It is worth pointing out that the converse of the above fact is not true in general, see for instance \cite{Dabkowski}. However, if $M$ is compact, then we have that ${\rm{scal}}(g) = 2s_{C}(g)$ if and only if ${\rm{d}\Omega = 0}$. Given a compact Hermitian manifold $(M,g,J)$, it will be useful for us to consider the following characterization of ${\rm{scal}}(g)$: 
\begin{equation}
\label{scalrelation}
{\rm{scal}}(g) = 2s_{C}(g) + \Big (\big \langle \partial \partial^{\ast}\Omega +  \bar{\partial} \bar{\partial}^{\ast}\Omega, \Omega \big \rangle - 2|\partial^{\ast}\Omega|^{2}\Big) - \frac{1}{2}|\mathcal{T}_{\nabla}|^{2},
\end{equation}
where\footnote{ $\langle \alpha, \beta \rangle \frac{\Omega^{n}}{n!} = \alpha \wedge \ast \bar{\beta}$, $\forall \alpha, \beta \in \Omega^{p,q}(M)$.} $|\mathcal{T}_{\nabla}|^{2} =  \big \langle \sqrt{-1}\Lambda(\partial \bar{\partial}\Omega_{\zeta}),\Omega_{\zeta} \big \rangle + \big \langle \partial \partial^{\ast}\Omega_{\zeta} +  \bar{\partial} \bar{\partial}^{\ast}\Omega_{\zeta}, \Omega_{\zeta} \big \rangle$, see \cite[Corollary 4.2]{LiuYang}.

\subsection{Calabi-Yau cones} Regarding $r \in (0,+\infty)$ as a coordinate on the positive real line $\mathbbm{R}_{+}$, we consider the following definition.

\begin{definition}
A Riemannian manifold $(Q,g_{Q})$ is Sasakian if and only if its metric cone $(\mathcal{C}(Q):= Q \times \mathbbm{R}_{+},g_{\mathcal{C}} := r^{2}g + {\rm{d}}r\otimes {\rm{d}}r)$ is a K\"{a}hler cone.
\end{definition}

Given a Sasakian manifold $(Q,g)$, it follows that $\dim_{\mathbbm{R}}(Q) = 2n+1$. Denoting by $\mathcal{J}$ the associated complex structure on $(\mathcal{C}(Q),g_{\mathcal{C}})$, we have the following facts:
\begin{enumerate}
\item the vector field $r \partial_{r}$ is real holomorphic, i.e. $\mathscr{L}_{r \partial_{r}}\mathcal{J} = 0;$
\item the (Reeb) vector field $\xi = \mathcal{J}(r\partial_{r})$ is real holomorphic and Killing, i.e.
\begin{center}
$\mathscr{L}_{\xi}\mathcal{J} = 0$ \ \ and \ \ $\mathscr{L}_{\xi}g_{\mathcal{C}} = 0;$ 
\end{center}
\item the $1$-form $\eta = {\rm{d}}^{c}\log(r)$, such that ${\rm{d}}^{c} = \mathcal{J} \circ {\rm{d}}$, satisfies the following 
\begin{center}
$\eta(\xi) = 1$ \ \ and \ \ $\xi \lrcorner {\rm{d}}\eta = 0$.
\end{center}
\end{enumerate}
From above, one can describe the K\"{a}hler form $\omega_{\mathcal{C}} = g_{\mathcal{C}}(\mathcal{J} \otimes \mathbbm{1})$ as follows
\begin{equation}
\omega_{\mathcal{C}} = \frac{1}{4}{\rm{d}}{\rm{d}}^{c}r^{2} = {\rm{d}}\Big ( \frac{r^{2}\eta}{2}\Big ).
\end{equation}
By considering the natural inclusion $Q \hookrightarrow \mathcal{C}(Q)$, such that $Q = \{r = 1\}$, we can consider $\eta$ and $\xi$ as tensor fields on $Q$. From this, since $\omega_{\mathcal{C}}$ is a symplectic structure on $\mathcal{C}(Q)$, it follows that $({\rm{d}}\eta)^{n} \wedge \eta \neq 0$ on $Q$. Therefore, we have that $(Q,\eta,\xi)$ is a contact manifold. Further, denoting $\mathcal{D} = \ker(\eta)$ and $\mathcal{F}_{\xi} = \big \langle \xi \big \rangle$, we can define $\Phi \in {\rm{End}}(TQ)$, such that $\Phi =  \mathcal{J}|_{\mathcal{D}}$ and $\Phi|_{\mathcal{F}_{\xi}} = 0$. From this, we have
\begin{enumerate}
\item $\mathcal{J}^{2} = - \mathbbm{1} \Longrightarrow \Phi^{2} = -\mathbbm{1} + \eta \otimes \xi$;
\item $g_{\mathcal{C}}(\mathcal{J} \otimes \mathbbm{1}) = - g_{\mathcal{C}}(\mathbbm{1} \otimes \mathcal{J}) \Longrightarrow g_{Q}(\Phi \otimes \Phi) = g_{Q} - \eta \otimes \eta$;
\item $[\mathcal{J},\mathcal{J}] = 0 \Longrightarrow [\Phi,\Phi] + {\rm{d}}\eta \otimes \xi = 0.$
\end{enumerate}
By using the above relations one can show that
\begin{equation}
g = \frac{1}{2}{\rm{d}}\eta(\mathbbm{1} \otimes \Phi) + \eta \otimes \eta.
\end{equation}
Moreover, we have that $g|_{\mathcal{D}} = \frac{1}{2}{\rm{d}}\eta(\mathbbm{1} \otimes \Phi)$ defines a Hermitian metric on $\mathcal{D}$. From this, denoting $g|_{\mathcal{D}} = g^{T}$, we have a (transverse) K\"{a}hler foliation $(g^{T},\mathcal{J}|_{\mathcal{D}},\mathcal{D})$ on $Q$. 
\begin{remark}
In the above setting, we shall refer to $\mathcal{S} := (\xi,\eta, \Phi,g)$ as being a Sasakian structure on $Q$. 
\end{remark}

\begin{remark}
\label{regularity}
Denoting also by $\mathcal{F}_{\xi}$ the Reeb foliation on $Q$ defined by $\xi$, unless otherwise stated, we will assume that all orbits of $\xi$ are all compact. Under this assumption, $\xi$ integrates to an isometric ${\rm{U}}(1)$ action on $(Q,g)$. We call the Sasakian structure $\mathcal{S}$ {\textit{quasi-regular}} if the ${\rm{U}}(1)$ action is locally free. If the ${\rm{U}}(1)$ action is free, we call the Sasakian structure $\mathcal{S}$ {\textit{regular}}. In the regular or quasi-regular case, the leaf space $X:= Q/\mathcal{F}_{\xi}$ has the structure of a manifold or orbifold, respectively. In both cases the transverse K\"{a}hler structure $(g^{T},\mathcal{J}|_{\mathcal{D}},\mathcal{D})$ pushes down to a K\"{a}hler structure on $X$. Moreover, we have the following facts:
\begin{enumerate}
\item $Q$ is the total space of a principal ${\rm{U}}(1)$-(orbi)bundle over $X$;
\item $\pi \colon Q \to X= Q/\mathcal{F}_{\xi}$ is a  (an orbifold) Riemannian submersion; 
\item ${\rm{d}}\eta = \pi^{\ast}\omega_{X}$, where $\omega_{X}$ is a non-trivial integral (orbifold) cohomology class; 
\item $(X,\omega_{X})$ is a K\"{a}hler (orbifold) manifold.
\end{enumerate}
\end{remark}

\begin{proposition}
Let $(Q,g)$ be a Sasakian manifold of dimension $2n+1$. Then the following are equivalent
\begin{enumerate}
\item $(Q,g)$ is Sasaki-Einstein with ${\rm{Ric}}(g) = 2ng$;
\item The K\"{a}hler cone $({\mathcal{C}}(Q),g_{\mathcal{C}})$ is Ricci-flat, ${\rm{Ric}}(g_{\mathcal{C}}) = 0$;
\item The transverse K\"{a}hler structure to the Reeb foliation $\mathcal{F}_{\xi}$ is K\"{a}hler-Einstein with ${\rm{Ric}}(g^{T}) = 2(n+1)g^{T}$.
\end{enumerate}
\end{proposition}
By a slight abuse of notation, we have the following
\begin{equation}
{\rm{Ric}}(g_{\mathcal{C}}) = {\rm{Ric}}(g) - 2ng = {\rm{Ric}}(g^{T}) - 2(n+1)g^{T}.
\end{equation}
It follows from Remark \ref{regularity} that every (quasi-)regular Sasaki-Einstein manifold is a principal ${\rm{U}}(1)$-(orbi)bundle over a K\"{a}hler-Einstein (orbifold) manifold $(X,\omega_{X})$. One can also obtain Sasaki-Einstein manifolds from  K\"{a}hler-Einstein (orbifold) manifolds, we shall explore later this construction in the regular case (Example \ref{fundexample}).

\subsection{Hermitian Weyl-Einstein manifolds} Let $(M,g,J)$ be a connected complex Hermitian manifold, such that $\dim_{\mathbbm{C}}(M) \geq 2$.

\begin{definition}
\label{DEFLCK}
A Hermitian manifold $(M,g,J)$ is called locally conformally K\"{a}hler (l.c.K.) if it satisfies one of the following equivalent conditions: 

\begin{enumerate}

\item There exists an open cover $\mathscr{U}$ of $M$ and a family of smooth functions $f_{U} \colon U \to \mathbbm{R}$, $U \in \mathscr{U}$, such that $g_{U} := {\mathrm{e}}^{-f_{U}}g|_{U}$, is K\"{a}hlerian, $\forall U \in \mathscr{U}$.

\item There exists a globally defined closed $1$-form $\theta \in \Omega^{1}(M)$, such that 
\begin{equation}
\label{LCKform}
{\rm{d}}\Omega = \theta \wedge \Omega.
\end{equation}
\end{enumerate}
\end{definition}

\begin{remark}
In the setting above, it follows from Definition \ref{LeeDef} that $\theta$ satisfying Eq. (\ref{LCKform}) is the Lee form associated to the Hermitian structure of a l.c.K. manifold $(M,g,J)$. Throughout this paper, unless otherwise stated, we shall assume for all l.c.K. manifold $(M,g,J)$ that $\theta$ is not exact and $\theta \not \equiv 0$.
\end{remark}

An important subclass of l.c.K. manifolds is defined by the parallelism of the Lee form with respect to the Levi-Civita connection of $g$. Being more precise, we have the following definition.

\begin{definition}
A l.c.K. manifold $(M,g,J)$ is called a Vaisman manifold if $\nabla \theta = 0$, where $\nabla$ is the Levi-Civita connection of $g$.
\end{definition}

\begin{remark}[\cite{Vaisman}, \cite{Dragomir}]
\label{Vaismandesc}
Since the Lee form of a Vaisman manifold is parallel, it has constant norm. Thus, the underlying Hermitian metric can be rescaled such that $||\theta||_{g} = 1$. The fundamental form $\Omega$ of the Vaisman metric with unit length Lee form satisfies the equality:
\begin{equation}
\label{FundVaisman}
\Omega = -{\rm{d}}(J\theta) + \theta \wedge J\theta.
\end{equation}
In the above description, we have that $-{\rm{d}}(J\theta) \geq 0$,  e.g. \cite[Theorem 5.1]{Dragomir}.
\end{remark}

\begin{example}
Let $(\mathcal{C}(Q),g_{\mathcal{C}},\mathcal{J})$ be the K\"{a}hler cone of a Sasaki manifold $(Q,g_{Q})$. By considering the identification $ \mathcal{C}(Q) \cong Q \times \mathbbm{R}$ defined by $\varphi = \log(r)$, it follows that $g_{\mathcal{C}} = {\rm{e}}^{2\varphi}\big (g_{Q} + {\rm{d}}\varphi \otimes  {\rm{d}}\varphi \big )$. From above, we consider the Hermitian manifold $(\mathcal{C}(Q),{\rm{e}}^{-2\varphi}g_{\mathcal{C}},\mathcal{J})$. Given a Sasaki automorphism $\phi \colon Q \to Q$ and $c > 0$, let $\Gamma_{\phi,c}$ be the cyclic group defined by
\begin{center}
$\Gamma_{\phi,c} := \Big \langle (x,\varphi) \mapsto (\phi(x),\varphi + \log(c)) \Big \rangle$.
\end{center}
Since $\Gamma_{\phi,c}$ acts by holomorphic isometries on $(\mathcal{C}(Q),{\rm{e}}^{-2\varphi}g_{\mathcal{C}},\mathcal{J})$, the Hermitian structure $({\rm{e}}^{-2\varphi}g_{\mathcal{C}},\mathcal{J})$ descends to a Hermitian structure $(g,J)$ on $\Sigma_{\phi,c}(Q) = \mathcal{C}(Q)/\Gamma_{\phi,c}$. By construction, since $g \myeq {\rm{e}}^{-2\varphi}g_{\mathcal{C}}$, it follows that $(\Sigma_{\phi,c}(Q) ,g,J)$ defines a l.c.K. manifold with Lee form described by $\theta \myeq -2{\rm{d}}\varphi$, see for instance \cite{Gini}. It is straightforward to show that $\theta$ is also parallel with respect to the Levi-Civita connection of $g$. Thus, we have that $(\Sigma_{\phi,c}(Q) ,g,J)$ is a Vaisman manifold.
\end{example}

\begin{remark}
In the previous example we have the following identification
\begin{equation}
\Sigma_{\phi,c}(Q) \cong \frac{Q \times [0,\log(c)]}{(\phi(x),0) \sim (x,\log(c))},
\end{equation}
i.e. $\Sigma_{\phi,c}(Q)$ can be seen as the suspension of $Q$ by $\phi$ over the circle of length $2\pi \log(c)$, e.g. \cite{Bazzoni}. We shall refer to $\Sigma_{\phi,c}(Q)$ as the suspension by $(\phi,c)$ of $Q$.
\end{remark}

Given a conformal manifold $(M,[g])$, we have the following notion of compatible connection with the conformal class $ [g]$.

\begin{definition}
\label{weylconnectiondef}
A Weyl connection $D$ on a conformal manifold $(M,[g])$ is a torsion-free connection which preserves the conformal class $[g]$. In this last setting, we say that $D$ defines a Weyl structure on $(M,[g])$ and that $(M,[g],D)$ is a Weyl manifold.
\end{definition}

In the above definition by preserving the conformal class it means that for each $g' \in [g]$, we have a $1$-form $\theta_{g'}$ (Higgs field), such that 
\begin{equation}
\label{higgs}
Dg' = \theta_{g'} \otimes g'.
\end{equation}
Let $(M,[g],D)$ be a Weyl manifold, in what follows we shall fix a representative $g$ for the underlying conformal class and consider the $1$-form $\theta_{g}$ which defines its Higgs field.
\begin{definition}
We say that a Weyl manifold $(M,[g],D)$ is a Hermitian-Weyl manifold if it admits an almost complex structure $J \in {\text{End}}(TM)$, which satisfies:

\begin{enumerate}

\item $g(JX,JY) = g(X,Y)$, $\forall X,Y \in \mathfrak{X}(M);$

\item $DJ = 0.$

\end{enumerate}

\end{definition}

An important result to be considered in the setting of Hermitian-Weyl manifolds is the following proposition.

\begin{proposition}[Vaisman]
\label{HweylVaisman}
Any Hermitian-Weyl manifold of (real) dimension at least $6$ is l.c.K.. Conversely, a l.c.K. manifold of (real) dimension at least $4$ admits a compatible Hermitian-Weyl structure.
\end{proposition}

For a compact Hermitian-Weyl manifold $(M,[g],D,J)$ Gauduchon \cite{Gauduchon} showed that, up to homothety, there is a unique choice of metric $g_{0}$ in the conformal class $[g]$ such that the corresponding $1$-form $\theta_{g_{0}}$ is co-closed. 

\begin{definition}
The unique (up to homothety) l.c.K. metric $g_{0}$ in the conformal class $[g]$ of $(M,[g],D,J)$ with harmonic associated Lee form is called the Gauduchon metric.
\end{definition}

It follows from Proposition \ref{HweylVaisman} that any compact Vaisman manifold admits a Hermitian-Weyl structure uniquely determined by the Gauduchon metric.

\begin{definition}
A Hermitian-Weyl manifold is Hermitian Weyl-Einstein if the symmetric part of the Ricci tensor of the Weyl connection is proportional to the metric. 
\end{definition}
In the setting of the above definition, it can be shown that the Hermitian Weyl-Einstein condition is equivalent to 
\begin{equation}
{\rm{Ric}}(g_{0}) = (n-2)\big (||\theta_{g_{0}}||^{2}g_{0} - \theta_{g_{0}} \otimes \theta_{g_{0}} \big),
\end{equation}
see for instance \cite{Gauduchon2}. From a deep result of Gauduchon in \cite{Gauduchon2}, it follows that:

\begin{theorem}
\label{HEW}
Let $(M,[g],D,J)$ be a compact Hermitian Weyl-Einstein manifold. Then the Ricci tensor of the Weyl connection vanishes identically and the Lee form is parallel. In particular, $(M,g_{0},J)$ is Vaisman. 
\end{theorem}

In the above setting, we have that $\theta_{g_{0}}$ is harmonic and ${\rm{Ric}}(g_{0})$ is non-negative, using the Weitzenb\"{o}ck formula, one can show that $b_{1}(M) = 1$. Now we consider the following structure theorem \cite{Verbitsky}.

\begin{theorem}
\label{isovaisman}
Every compact Vaisman manifold with $b_{1}(M) = 1$ is isomorphic to $(\Sigma_{\phi,c}(Q),g,J)$, where $Q$ is some compact Sasakian manifold.
\end{theorem}

Combining the result above with the previous comments, one can show that every compact Hermitian Weyl-Einstein manifold is isomorphic to  $(\Sigma_{\phi,c}(Q),g_{0},J)$ as a Vaisman manifold, where $Q$ is a Sasaki-Einstein manifold. Thus, in the setting of Theorem \ref{HEW} we have $g_{0} \myeq {\rm{e}}^{-2\varphi}g_{CY}$, where $g_{CY}$ is a Calabi-Yau metric on the metric cone $\mathcal{C}(Q)$ of a Sasaki-Einstein manifold $Q$.

\section{Proof of main results}
\label{S3}

\begin{lemma}
\label{lemmafund}
Let $(M,\Omega,J)$ be a compact l.c.K. manifold. Then 
\begin{equation}
\label{LCRFcohomology}
 \mathfrak{R}^{(1)}(\Omega) =  \Upsilon - \frac{1}{2}{\rm{d}}(J\theta),
\end{equation}
where $\theta$ is the Lee form of $(M,\Omega,J)$ and $\Upsilon \in 2\pi c_{1}(M)$. In particular, if $\mathfrak{R}^{(1)}(\Omega)=0$, then $c_{1}(M) = 0$.
\end{lemma}

\begin{proof}
Given a l.c.K. manifold $(M,\Omega,J)$, it follows that $\Omega \myeq {\rm{e}}^{f_{U}}\omega_{U}$, thus
\begin{equation}
{\rm{Ric}}^{(1)}(\Omega) \myeq  -\frac{1}{2}{\rm{d}}{\rm{d}}^{c}\log(\det(\omega_{U})) - \frac{n}{2}{\rm{d}}{\rm{d}}^{c}f_{U}.
\end{equation}
Since, $\theta \myeq {\rm{d}}f_{U}$, it follows that ${\rm{Ric}}^{(1)}(\Omega) \myeq  -\frac{1}{2}{\rm{d}}{\rm{d}}^{c}\log(\det(\omega_{U})) - \frac{n}{2}{\rm{d}}(J\theta)$. Thus, gluing $-\frac{1}{2}{\rm{d}}{\rm{d}}^{c}\log(\det(\omega_{U}))$, we obtain a globally defined $(1,1)$-form $\Upsilon$, such that 
\begin{equation}
{\rm{Ric}}^{(1)}(\Omega) = \Upsilon - \frac{n}{2}{\rm{d}}(J\theta).
\end{equation}
In particular, we have $\Upsilon \in 2\pi c_{1}(M)$. Thus, if $(M,\Omega,J)$ is a compact l.c.K. manifold, it follows that 
\begin{equation}
\label{LCRLCK}
\mathfrak{R}^{(1)}(\Omega) =  \Upsilon - \frac{n}{2}{\rm{d}}(J\theta)- \frac{1}{2} \big ( \partial \partial^{\ast}\Omega +  \bar{\partial} \bar{\partial}^{\ast} \Omega\big).
\end{equation}
Considering the Hodge decomposition $\theta = \theta^{1,0} + \theta^{0,1}$, we observe the following:
\begin{enumerate}
\item ${\rm{d}}\theta = 0 \Rightarrow \bar{\partial} \theta^{1,0} = - \partial \theta^{0,1}, \ \ \partial \theta^{1,0} = \bar{\partial} \theta^{0,1} = 0$;
\item $\bar{\partial} \theta^{1,0} = - \partial \theta^{0,1} \iff \bar{\partial} \tau = -\partial \bar{\tau} \Rightarrow {\rm{d}}(J\theta) = -\frac{2}{n-1}\sqrt{-1}\bar{\partial}\tau$;
\item $\partial \partial^{\ast}\Omega +  \bar{\partial} \bar{\partial}^{\ast}\Omega = -\sqrt{-1}\partial \bar{\tau} + \sqrt{-1}\bar{\partial} \tau = 2\sqrt{-1}\bar{\partial} \tau = -(n-1){\rm{d}}(J\theta)$.
\end{enumerate}
By using the above relations in Eq. (\ref{LCRLCK}), we conclude that 
\begin{equation}
\mathfrak{R}^{(1)}(\Omega) =  \Upsilon - \frac{1}{2}{\rm{d}}(J\theta).
\end{equation}
From this, it follows that $[\mathfrak{R}^{(1)}(\Omega)] \in 2\pi c_{1}(M)$. In particular, if $\mathfrak{R}^{(1)}(\Omega)=0$, we have $c_{1}(M)=0$. In this last case, Eq. (\ref{LCRFcohomology}) holds for $\Upsilon = \frac{1}{2}{\rm{d}}(J\theta)$.
\end{proof}

\begin{remark}
In the setting of the above lemma, the fact that $[\mathfrak{R}^{(1)}(\Omega)] \in 2\pi c_{1}(M)$ also can be seen as a consequence of \cite[Theorem 3.14, item (2)]{LiuYang}. In fact, since $\bar{\partial}\theta^{0,1} = 0$, it follows that $\bar{\partial}\partial^{\ast}\Omega = -\sqrt{-1}\bar{\partial}\bar{\tau} =  -\sqrt{-1}(n-1)\bar{\partial}\theta^{0,1} = 0.$

\end{remark}

\begin{theorem}
\label{mainproof}
Let $(Q,g_{SE})$ be a compact Sasaki-Einstein manifold, $\phi \colon Q \to Q$ a Sasaki automorphism, and $c > 0$. Then, the suspension $\Sigma_{\phi,c}(Q)$ by $(\phi,c)$ of $Q$ admits a Levi-Civita Ricci-flat Hermitian metric.
\end{theorem}

\begin{proof}
Given a Sasaki-Einstein manifold $(Q,g_{SE})$, it follows that the K\"{a}hler cone $(\mathcal{C}(Q),\omega_{CY} :=\frac{1}{4}{\rm{d}}{\rm{d}}^{c}r^{2} ,\mathcal{J})$ is Calabi-Yau. Given a Sasaki automorphism $\phi \colon Q \to Q$ and $c > 0$, we have that the Hermitian structure $(\mathcal{C}(Q),{\rm{e}}^{-2\varphi}\omega_{CY},\mathcal{J})$, such that $\varphi = \log(r)$, descends to a Vaisman structure $(\Omega,J)$ on the suspension $\Sigma_{\phi,c}(Q)$. From Lemma \ref{lemmafund}, it follows that 
\begin{equation}
\mathfrak{R}^{(1)}(\Omega) =  \Upsilon - \frac{1}{2}{\rm{d}}(J\theta).
\end{equation}
By considering the projection $\wp \colon \mathcal{C}(Q) \to \Sigma_{\phi,c}(Q)$, since $\Omega \myeq {\rm{e}}^{-2\varphi}\omega_{CY}$, it follows that $\wp^{\ast}\Upsilon = {\rm{Ric}}^{(1)}(\omega_{CY}) = 0$. Thus, we have $\mathfrak{R}^{(1)}(\Omega) = -\frac{1}{2}{\rm{d}}(J\theta)$. Given $\zeta  > -1$, we consider the perturbed Hermitian metric $\Omega_{\zeta}$ on $\Sigma_{\phi,c}(Q)$ given by
\begin{equation}
\label{Hermitianfamily}
\Omega_{\zeta} := \Omega + 2\zeta\mathfrak{R}^{(1)}(\Omega) = \Omega - \zeta{\rm{d}}(J\theta).  
\end{equation}
Since $\Omega$ is Vaisman, it follows from Eq. (\ref{FundVaisman}) that 
\begin{equation}
\Omega_{\zeta} = -(1+\zeta){\rm{d}}(J\theta) + \theta \wedge J\theta \Rightarrow {\rm{d}}\Omega_{\zeta} = \theta_{\zeta} \wedge \Omega_{\zeta},
\end{equation}
such that $\theta_{\zeta} = \frac{1}{1+\zeta}\theta$. By construction, we have $\Omega_{\zeta}^{n} = (1+\zeta)^{n-1}\Omega^{n}$, thus
\begin{equation}
{\rm{Ric}}^{(1)}(\Omega_{\zeta}) = {\rm{Ric}}^{(1)}(\Omega) = \Upsilon - \frac{n}{2}{\rm{d}}(J\theta) = - \frac{n}{2}{\rm{d}}(J\theta).
\end{equation}
By considering the Hodge $\ast$-operator induced by $\Omega_{\zeta}$, it follows that 
\begin{equation}
\partial \partial^{\ast}\Omega_{\zeta} +  \bar{\partial} \bar{\partial}^{\ast}\Omega_{\zeta} = -(n-1){\rm{d}}(J\theta_{\zeta}) = -\bigg (\frac{n-1}{1+\zeta}\bigg){\rm{d}}(J\theta).
\end{equation}
Combining the above expressions, we obtain
\begin{equation}
\mathfrak{R}^{(1)}(\Omega_{\zeta}) := {\rm{Ric}}^{(1)}(\Omega_{\zeta}) - \frac{1}{2} \big ( \partial \partial^{\ast}\Omega_{\zeta} +  \bar{\partial} \bar{\partial}^{\ast} \Omega_{\zeta}\big) = \bigg (- n + \frac{n-1}{1+\zeta}\bigg){\rm{d}}(J\theta).
\end{equation}
Therefore, for $\zeta = -\frac{1}{n}$, we have that $(\Sigma_{\phi,c}(Q),\Omega_{\zeta},J)$ is Levi-Civita Ricci-flat.
\end{proof}

Since every 3-Sasakian manifold is Sasaki-Einstein (e.g. \cite{3Einstein}, \cite{Kashiwada2}) we obtain the following corollary.

\begin{corollary}
Let $(Q,g_{Q})$ be a compact $3$-Sasakian manifold, $\phi \colon Q \to Q$ a Sasaki automorphism, and $c > 0$. Then, the suspension $\Sigma_{\phi,c}(Q)$ by $(\phi,c)$ of $Q$ admits a Levi-Civita Ricci-flat Hermitian metric.
\end{corollary}

From \cite{Verbitsky} and Theorem \ref{mainproof}, we have the following result.

\begin{corollary}
Every compact Hermitian Weyl-Einstein manifold admits a Levi-Civita Ricci-flat Hermitian metric. In particular, every compact locally conformal hyperK\"{a}hler manifold admits a Levi-Civita Ricci-flat Hermitian metric. 
\end{corollary}

Combining the result of Theorem \ref{mainproof} with \cite{LiuSanoTasin}, see also \cite[Conjecture 4]{Einsteinsphere}, we have the following corollaries.

\begin{corollary}
\label{LCRFBP}
Let ${\rm{L}}({\bf{a}}) := Y({\bf{a}}) \cap S^{2n+1}$ be the link of a Brieskorn-Pham singularity
\begin{equation}
Y({\bf{a}}) := \Big (  z_{0}^{a_{0}} + \cdots + z_{n}^{a_{n}} = 0 \Big) \subset \mathbbm{C}^{n+1},
\end{equation}
such that $n \geq 3$. Assume that $a_{0} \leq \cdots \leq a_{n}$. Then ${\rm{L}}({\bf{a}}) \times S^{1}$ admits a Levi-Civita Ricci-flat Hermitian metric if 
\begin{equation}
1 < \sum_{j = 0}^{n}\frac{1}{a_{j}} < 1 + \frac{n}{a_{n}}.
\end{equation}
\end{corollary}

\begin{corollary}
\label{exoticsphereLCRF}
Let ${\bf{\Sigma}}$ be an odd dimensional homotopy sphere which bounds a parallelizable manifold. Then ${\bf{\Sigma}} \times S^{1}$ admits a Levi-Civita Ricci-flat Hermitian metric.
\end{corollary}

Our next result generalizes some ideas introduced in  \cite[Theorem 6.4]{LiuSanoTasin}. 

\begin{theorem}
\label{T2proof}
Let $(Q,g_{SE})$ be a compact Sasaki-Einstein manifold, $\phi \colon Q \to Q$ a Sasaki automorphism, and $c > 0$. Then, the suspension $\Sigma_{\phi,c}(Q)$ by $(\phi,c)$ of $Q$ admits three different Hermitian metrics $\Omega_{i}$, $i = 1,2,3$, satisfying the following properties:
\begin{enumerate}
\item[{\rm{(1)}}] ${\rm{Ric}}^{(1)}(\Omega_{1}) = {\rm{Ric}}^{(1)}(\Omega_{2}) = {\rm{Ric}}^{(1)}(\Omega_{3}) \geq 0$;
\item[{\rm{(2)}}] $\Omega_{1}$ has {\bf{strictly positive}} Riemannian scalar curvature;
\item[{\rm{(3)}}] $\Omega_{2}$ has {\bf{zero}} Riemannian scalar curvature;
\item[{\rm{(4)}}] $\Omega_{3}$ has {\bf{strictly negative}} Riemannian scalar curvature.
\end{enumerate}
In particular, all compact Hermitian manifolds of the previous corollaries admit three different Hermitian metrics satisfying the above properties.
\end{theorem}

\begin{proof}
For every $\zeta > -1$, let $(\Sigma_{\phi,c}(Q),\Omega_{\zeta},J)$ be the Hermitian compact manifold constructed as in the proof of Theorem \ref{mainproof}. It follows from Eq. (\ref{scalrelation}) that the Riemannian scalar curvature of the underlying Riemannian metric $g_{\zeta} = \Omega_{\zeta}(\mathbbm{1} \otimes J)$ is given by 
\begin{equation}
\label{Rscal}
{\rm{scal}}(g_{\zeta}) = 2s_{C}(g_{\zeta}) + \Big ( \big \langle \partial \partial^{\ast}\Omega_{\zeta} +  \bar{\partial} \bar{\partial}^{\ast}\Omega_{\zeta}, \Omega_{\zeta} \big \rangle - 2|\partial^{\ast}\Omega_{\zeta}|^{2}\Big) - \frac{1}{2}|\mathcal{T}_{\nabla}|^{2},
\end{equation}
where $|\mathcal{T}_{\nabla}|^{2} =  \big \langle \sqrt{-1}\Lambda(\partial \bar{\partial}\Omega_{\zeta}),\Omega_{\zeta} \big \rangle + \big \langle \partial \partial^{\ast}\Omega_{\zeta} +  \bar{\partial} \bar{\partial}^{\ast}\Omega_{\zeta}, \Omega_{\zeta} \big \rangle$. Since we have
\begin{equation}
\Omega_{\zeta} = -(1+\zeta){\rm{d}}(J\theta) + \theta \wedge J\theta, \ \ \ {\rm{d}}\Omega_{\zeta} = \theta_{\zeta} \wedge \Omega_{\zeta},
\end{equation}
such that $\theta_{\zeta} = \frac{1}{1+\zeta}\theta$, in order to compute ${\rm{scal}}(g_{\zeta})$, it will be useful to consider the following identities:
\begin{enumerate}
\item[(A)] $\displaystyle {\rm{vol}} := \frac{1}{n!}\Omega_{\zeta}^{n} = \frac{(-1)^{n-1}(\zeta+1)^{n-1}}{(n-1)!}({\rm{d}}(J\theta))^{n-1} \wedge \theta \wedge J\theta$;
\item[(B)] ${\rm{d}}(J\theta) \wedge \Omega_{\zeta}^{n-1} = (-1)^{n-2}(n-1)(\zeta + 1)^{n-2}({\rm{d}}(J\theta))^{n-1} \wedge \theta \wedge J\theta$;
\item[(C)] $\theta  \wedge J\theta \wedge \Omega_{\zeta}^{n-1} = (-1)^{n-1}(\zeta + 1)^{n-1} ({\rm{d}}(J\theta))^{n-1} \wedge \theta \wedge J\theta.$
\end{enumerate}
Notice that, from (A), (B), and (C), it follows that
\begin{equation}
\label{fundrelations}
{\rm{d}}(J\theta) \wedge \Omega_{\zeta}^{n-1} = -(n-1)!\bigg(\frac{n-1}{\zeta+1}\bigg){\rm{vol}} \ \ \ {\text{and}} \ \ \ \theta  \wedge J\theta \wedge \Omega_{\zeta}^{n-1} = (n-1)!{\rm{vol}}.
\end{equation}
Since $\partial \partial^{\ast}\Omega_{\zeta} +  \bar{\partial} \bar{\partial}^{\ast}\Omega_{\zeta} = -(n-1){\rm{d}}(J\theta_{\zeta})$, it follows that  
\begin{equation}
 \big \langle \partial \partial^{\ast}\Omega_{\zeta} +  \bar{\partial} \bar{\partial}^{\ast}\Omega_{\zeta}, \Omega_{\zeta} \big \rangle {\rm{vol}} =  -(n-1){\rm{d}}(J\theta_{\zeta}) \wedge \ast \Omega_{\zeta} = -\bigg (\frac{n-1}{\zeta + 1}\bigg ){\rm{d}}(J\theta) \wedge \ast \Omega_{\zeta}.
\end{equation}
As $\ast \Omega_{\zeta} = \frac{\Omega_{\zeta}^{n-1}}{(n-1)!}$, it follow from Eq. (\ref{fundrelations}) that 
\begin{equation}
\label{ID0}
-\bigg (\frac{n-1}{\zeta + 1}\bigg ){\rm{d}}(J\theta) \wedge \ast \Omega_{\zeta} =  \bigg ( \frac{n-1}{\zeta+1} \bigg )^{2}{\rm{vol}}.
\end{equation}
Therefore, we conclude that 
\begin{equation}
\label{ID1}
 \big \langle \partial \partial^{\ast}\Omega_{\zeta} +  \bar{\partial} \bar{\partial}^{\ast}\Omega_{\zeta}, \Omega_{\zeta} \big \rangle = \bigg ( \frac{n-1}{\zeta+1} \bigg )^{2}. 
\end{equation}
Now we observe that $|\partial^{\ast}\Omega_{\zeta}|^{2} = \big \langle \partial^{\ast}\Omega_{\zeta},\partial^{\ast}\Omega_{\zeta}\big \rangle =  \big \langle \partial \partial^{\ast}\Omega_{\zeta},\Omega_{\zeta}\big \rangle$. Since
\begin{equation}
\partial \partial^{\ast}\Omega_{\zeta} = -\frac{(n-1)}{2}{\rm{d}}(J\theta_{\zeta}) = -\frac{1}{2}\bigg (\frac{n-1}{\zeta + 1}\bigg ){\rm{d}}(J\theta),
\end{equation}
from a similar argument as in Eq. (\ref{ID0}), we have 
\begin{equation}
\big \langle \partial^{\ast}\Omega_{\zeta},\partial^{\ast}\Omega_{\zeta}\big \rangle {\rm{vol}} = - \frac{1}{2}\bigg (\frac{n-1}{\zeta + 1}\bigg ){\rm{d}}(J\theta) \wedge \ast \Omega_{\zeta} = \frac{1}{2}\bigg (\frac{n-1}{\zeta + 1}\bigg )^{2}{\rm{vol}}.
\end{equation}
Hence, we obtain 
\begin{equation}
\label{ID2}
|\partial^{\ast}\Omega_{\zeta}|^{2} = \frac{1}{2}\bigg (\frac{n-1}{\zeta + 1}\bigg )^{2}.
\end{equation}
Replacing Eq. (\ref{ID1}) and Eq. (\ref{ID2}) in Eq. (\ref{Rscal}), we obtain
\begin{equation}
{\rm{scal}}(g_{\zeta}) = 2s_{C}(g_{\zeta}) - \frac{1}{2}|\mathcal{T}_{\nabla}|^{2}.
\end{equation}
Since ${\rm{Ric}}^{(1)}(\Omega_{\zeta}) = - \frac{n}{2}{\rm{d}}(J\theta)$, it follows that 
\begin{equation}
 \frac{s_{C}(g_{\zeta})}{n} \Omega_{\zeta}^{n} = {\rm{Ric}}^{(1)}(\Omega_{\zeta}) \wedge \Omega_{\zeta}^{n-1} = - \frac{n}{2}{\rm{d}}(J\theta) \wedge \Omega_{\zeta}^{n-1}.  
\end{equation}
From Eq. (\ref{fundrelations}), we obtain 
\begin{equation}
\label{Cscalar}
 \frac{s_{C}(g_{\zeta})}{n}\Omega_{\zeta}^{n} = \frac{n!}{2}\bigg(\frac{n-1}{\zeta+1}\bigg){\rm{vol}} \iff s_{C}(g_{\zeta}) = \frac{n(n-1)}{2(\zeta + 1)}.
\end{equation}
In order to describe ${\rm{scal}}(g_{\zeta})$, it remains to compute $\frac{1}{2}|\mathcal{T}_{\nabla}|^{2}$. From Eq. (\ref{ID1}), we have 
\begin{equation}
|\mathcal{T}_{\nabla}|^{2} = \big \langle \sqrt{-1}\Lambda(\partial \bar{\partial}\Omega_{\zeta}),\Omega_{\zeta} \big \rangle + \bigg ( \frac{n-1}{\zeta+1} \bigg )^{2}.
\end{equation}
We notice that
\begin{equation}
\big \langle \sqrt{-1}\Lambda(\partial \bar{\partial}\Omega_{\zeta}),\Omega_{\zeta} \big \rangle = \big \langle \sqrt{-1}\partial \bar{\partial}\Omega_{\zeta},{\rm{L}}(\Omega_{\zeta}) \big \rangle = \big \langle \sqrt{-1}\partial \bar{\partial}\Omega_{\zeta},\Omega_{\zeta}^{2} \big \rangle.
\end{equation}
By using that ${\rm{d}}\Omega_{\zeta} = \theta_{\zeta} \wedge \Omega_{\zeta}$ and ${\rm{d}}\theta_{\zeta} = 0$, we obtain 
\begin{equation}
\sqrt{-1}\partial \bar{\partial}\Omega_{\zeta} = \frac{1}{2}{\rm{d}}(J \theta_{\zeta}) \wedge \Omega_{\zeta} + \frac{1}{2}\theta_{\zeta} \wedge J\theta_{\zeta} \wedge \Omega_{\zeta}.
\end{equation}
From above, by using that $ \ast \Omega_{\zeta}^{2} = \frac{2\Omega_{\zeta}^{n-2}}{(n-2)!}$, it follows that: 
\begin{enumerate}
\item[(i)] $\displaystyle \frac{1}{2}{\rm{d}}(J \theta_{\zeta}) \wedge \Omega_{\zeta} \wedge \ast \Omega_{\zeta}^{2} = \frac{{\rm{d}}(J \theta) \wedge \Omega_{\zeta}^{n-1}}{(n-2)!(\zeta+1)} = -  \bigg ( \frac{n-1}{\zeta+1} \bigg )^{2}{\rm{vol}}$;
\item[(ii)] $\displaystyle \frac{1}{2}\theta_{\zeta} \wedge J\theta_{\zeta} \wedge \Omega_{\zeta} \wedge \ast \Omega_{\zeta}^{2} = \frac{\theta  \wedge J\theta \wedge \Omega_{\zeta}^{n-1}}{(n-2)!(\zeta+1)^{2}} = \frac{(n-1)}{(\zeta+1)^{2}}{\rm{vol}}$.
\end{enumerate}
Here we have used Eq. (\ref{fundrelations}) to obtain the above expressions. Hence, we conclude that  
\begin{equation}
\big \langle \sqrt{-1}\Lambda(\partial \bar{\partial}\Omega_{\zeta}),\Omega_{\zeta} \big \rangle = \big \langle \sqrt{-1}\partial \bar{\partial}\Omega_{\zeta},\Omega_{\zeta}^{2} \big \rangle =   -  \bigg ( \frac{n-1}{\zeta+1} \bigg )^{2} + \frac{(n-1)}{(\zeta+1)^{2}}.
\end{equation}
Therefore, we have 
\begin{equation}
|\mathcal{T}_{\nabla}|^{2} = \frac{(n-1)}{(\zeta+1)^{2}}.
\end{equation}
Combining the above expression with Eq. (\ref{Cscalar}), we obtain
\begin{equation}
\label{scalarfunction}
{\rm{scal}}(g_{\zeta}) = \frac{n(n-1)}{(\zeta + 1)} - \frac{(n-1)}{2(\zeta+1)^{2}} = \frac{n(n-1)}{(\zeta + 1)^{2}} \bigg [ \zeta - \frac{(1-2n)}{2n} \bigg ].
\end{equation}
From this, we have the following:
\begin{enumerate}
\item[(1)] $\zeta > \frac{(1-2n)}{2n} \Rightarrow \Omega_{\zeta}$ has strictly positive constant Riemannian scalar curvature;
\item[(2)] $\zeta = \frac{(1-2n)}{2n}  \Rightarrow \Omega_{\zeta}$ has constant zero Riemannian scalar curvature;
\item[(3)] $\zeta < \frac{(1-2n)}{2n} \Rightarrow \Omega_{\zeta}$ has strictly negative constant Riemannian scalar curvature.
\end{enumerate}
Hence, since ${\rm{Ric}}^{(1)}(\Omega_{\zeta}) = - \frac{n}{2}{\rm{d}}(J\theta) \geq 0$, $\forall \zeta > -1$, see Remark \ref{Vaismandesc}, we can always find three different Hermitian metrics $\Omega_{i}$, $i = 1,2,3$, satisfying the desired properties. The last statement of the theorem follows immediately from the above construction, and from the fact that all compact Hermitian manifolds mentioned can be obtained as a suspension $\Sigma_{\phi,c}(Q)$ of some suitable Sasaki-Einstein manifold $(Q,g_{Q})$.
\end{proof}

\section{Examples}

In this section, in order to illustrate the main results, we provide a general method to construct explicit examples of Levi-Civita Ricci-flat Hermitian metric from K\"{a}hler-Einstein Fano (orbifolds) manifolds. Also, we illustrate the results of Theorem \ref{T2proof} in the case that $Q$ is an exotic $7$-sphere.

\begin{example} 
\label{fundexample}
Let $(X,\omega_{X},J)$ be a K\"{a}hler-Einstein Fano manifold of complex dimension $n$ and Fano index ${\bf{I}}(X) \in \mathbbm{Z}$. Suppose that ${\rm{Ric}}(\omega_{X}) = \lambda\omega_{X}$, for some $\lambda > 0$. Considering ${\bf{L}} = {\bf{K}}_{X}^{\otimes \frac{\ell}{{\bf{I}}(X)}}$, for some $\ell \in \mathbbm{Z}_{>0}$, and let ${\bf{H}}$ be a Hermitian structure on ${\bf{L}}$, such that  
\begin{equation}
\label{curvaturerelation}
\frac{\sqrt{-1}}{2\pi}{\bf{\Theta}}(\nabla) = - \frac{\ell \lambda \omega_{X}}{2\pi {\bf{I}}(X)},
\end{equation}
where ${\bf{\Theta}}(\nabla)$ is the curvature of the associated Chern connection $\nabla  = {\rm{d}} + {\rm{d}}\log({\bf{H}})$. Considering the complex manifold  ${\rm{Tot}}({\bf{L}}^{\times})$ underlying the total space of the principal $\mathbbm{C}^{\times}$-bundle  ${\bf{L}}^{\times}:= {\bf{L}} - \{0{\text{-section}}\}$, let $r \colon {\rm{Tot}}({\bf{L}}^{\times}) \to \mathbbm{R}$, such that $r^{2} = {\bf{H}}(u,u)$, $\forall u \in {\rm{Tot}}(\bf{L}^{\times})$. Denoting by $\mathscr{J}$ the canonical complex structure on ${\rm{Tot}}({\bf{L}}^{\times})$, we have
\begin{equation}
\label{localform}
{\rm{d}}{\rm{d}}^{c}\log(r) = -p^{\ast}(\sqrt{-1}{\bf{\Theta}}(\nabla)) = p^{\ast} \bigg ( \frac{\ell \lambda \omega_{X}}{ {\bf{I}}(X)}\bigg),
\end{equation}
where $p \colon {\rm{Tot}}({\bf{L}}^{\times}) \to X$ is the associated bundle projection and ${\rm{d}}^{c} = \mathscr{J} \circ {\rm{d}}$. By considering the sphere bundle $Q({\bf{L}}) = \big \{u \in {\bf{L}} \ \big | \sqrt{{\bf{H}}(u,u)} = 1\big \}$ we have an identification $ {\rm{Tot}}({\bf{L}}^{\times}) \cong Q({\bf{L}}) \times \mathbbm{R}_{+}$ provide by the map
\begin{equation}
u \mapsto \Bigg ( \frac{u}{\sqrt{{\bf{H}}(u,u)}},\sqrt{{\bf{H}}(u,u)}\Bigg).
\end{equation}
Under the above identification, by considering the rescaled potential $\varrho = r^{\frac{{\bf{I}}(X)}{\ell(n+1)}}$ we have a K\"{a}hler structure $\omega_{CY}$ on $\mathcal{C}(Q({\bf{L}}))$ defined by
\begin{equation}
\omega_{CY} := \frac{1}{4}{\rm{d}}{\rm{d}}^{c}\varrho^{2} = \frac{{\bf{I}}(X)}{\ell(n+1)}{\rm{d}}\Big ( \frac{\varrho^{2}\eta}{2}\Big),
\end{equation}
such that $\eta = {\rm{d}}^{c} \log(r)$. We notice that 
\begin{center}
$\displaystyle{{\rm{d}}\log(\varrho)= \frac{{\bf{I}}(X)}{\ell(n+1)} {\rm{d}}\log(r) \Rightarrow \mathscr{J}\Big (\frac{{\rm{d}}\varrho}{\varrho}\Big) = \frac{{\bf{I}}(X)}{\ell(n+1)}\eta}$.
\end{center}
Thus, a straightforward computation shows that 
\begin{equation}
\label{CYcone}
g_{CY} = \omega_{CY}(\mathbbm{1} \otimes \mathscr{J}) = \varrho^{2}g_{SE} + {\rm{d}}\varrho \otimes {\rm{d}}\varrho,
\end{equation}
such that $g_{SE}$ is a Riemannian metric on $Q({\bf{L}})$ defined by 
\begin{equation}
g_{SE} =  \pi^{\ast}(g_{KE}) + \frac{{\bf{I}}(X)^{2}}{\ell^{2}(n+1)^{2}}\eta \otimes \eta,
\end{equation}
where $g_{KE} = \frac{\lambda}{2(n+1)}\omega_{X}(\mathbbm{1} \otimes J)$ and $\pi \colon Q({\bf{L}}) \to X$ is the associated projection. By construction, we have
\begin{center}
${\rm{Ric}}(g_{KE}) = 2(n+1)g_{KE} \iff {\rm{Ric}}(g_{SE}) = 2ng_{SE} \iff {\rm{Ric}}(g_{CY}) = 0$.
\end{center}
Therefore, we have that $(Q({\bf{L}}),g_{SE})$ is a regular Sasaki-Einstein manifold and its metric cone $(\mathcal{C}(Q({\bf{L}})),\omega_{CY})$ is Calabi-Yau. From this, given a Sasaki automorphism $\phi \colon Q({\bf{L}}) \to Q({\bf{L}})$, and $c > 0$, it follows from Theorem \ref{mainproof} that the suspension $\Sigma_{\phi,c}(Q({\bf{L}}))$ by $(\phi,c)$ of $Q({\bf{L}})$ admits a Levi-Civita Ricci-flat Hermitian metric $\Omega$ given by
\begin{center}
$\Omega = -\big (1-\frac{1}{n} \big){\rm{d}}(J\theta) + \theta \wedge J\theta$,
\end{center}
such that 
\begin{equation}
\theta \myeq -{\rm{d}}\log(\varrho^{2}) = -\frac{{\bf{I}}(X)}{\ell(m+1)}{\rm{d}}\log(r^{2}) = -\frac{{\bf{I}}(X)}{\ell(m+1)}{\rm{d}}\log({\bf{H}}). 
\end{equation}
Locally, we have ${\bf{H}} = h|w|^{2}$, such that $h \in C^{\infty}(U)$, for some open set $U \subset X$ which trivializes ${\bf{L}}$, satisfying $\frac{1}{2}{\rm{d}}{\rm{d}}^{c}\log(h) \myeq \frac{\ell \lambda \omega_{X}}{ {\bf{I}}(X)}$, see Eq. (\ref{localform}). Therefore, we conclude that 
\begin{equation}
\theta \myeq -\frac{{\bf{I}}(X)}{\ell(m+1)}\bigg [{\rm{d}}\log(h) + \frac{\bar{w}{\rm{d}}w +w{\rm{d}}\bar{w} }{|w|^{2}}\bigg]. 
\end{equation}
Summarizing, the Levi-Civita Ricci-flat metric $\Omega$ can be explicitly described in terms of the local potentials of the Chern connection $\nabla$ on ${\bf{L}}$ which satisfies Eq. (\ref{curvaturerelation}).
\end{example}

\begin{remark}
\label{remarkdeformcpx}
An alternative way to equip $\mathcal{C}(Q({\bf{L}}))$ with a Calabi-Yau metric is the following. In the setting of the previous example, consider $(Q({\bf{L}}),g)$, such that 
\begin{equation}
g=  \pi^{\ast} \bigg ( \frac{\ell \lambda g_{KE}}{ 2{\bf{I}}(X)}\bigg) + \eta \otimes \eta,
\end{equation}
where  $g_{KE} = \omega_{X}(\mathbbm{1} \otimes J)$. In this case, we have that $(Q({\bf{L}}),g)$ is a Sasakian manifold, and its Sasakian structure $\mathcal{S} = (\xi,\eta, \Phi,g)$ can be obtained from the K\"{a}hler cone $(\mathcal{C}(Q({\bf{L}})), \omega_{\mathcal{C}} = {\rm{d}} \big ( \frac{r^{2}\eta}{2}\big),\mathscr{J})$. Given $a > 0$, consider $\mathcal{S}_{a} =  (\frac{1}{a}\xi,a\eta, \Phi,g_{a})$, such that 
\begin{center}
$g_{a} = ag + (a^{2} - a)\eta \otimes \eta$.
\end{center}
By construction, we have that $(Q({\bf{L}}),g_{a})$ is a Riemannian manifold. From $\mathcal{S}_{a}$ we can construct a complex structure $\mathcal{J}_{a}$ on $\mathcal{C}(Q({\bf{L}})$ by setting 
\begin{center}
$\mathcal{J}_{a}(Y)= \phi(Y) - a\eta(Y)r\partial_{r}, \ \ \ \ \ \mathcal{J}_{a}( r\partial_{r}) = \frac{1}{a}\xi,$
\end{center}
$\forall Y \in \mathfrak{X}(Q({\bf{L}}))$. It follows that $(\mathcal{C}(Q({\bf{L}})), g_{\mathcal{C},a}:= r^{2}g_{a} + {\rm{d}}r\otimes {\rm{d}}r,\mathcal{J}_{a})$ is a K\"{a}hler manifold. Notice that $\omega_{\mathcal{C},a} = g_{\mathcal{C},a}(\mathcal{J}_{a} \otimes \mathbbm{1})$ is given by
\begin{equation}
\omega_{\mathcal{C},a} = \frac{1}{4}{\rm{d}}{\rm{d}}^{c}r^{2} = a{\rm{d}}\Big ( \frac{r^{2}\eta}{2}\Big ),
\end{equation}
such that ${\rm{d}}^{c}:= \mathcal{J}_{a} \circ {\rm{d}}$. Thus, we have that $(Q({\bf{L}}),g_{a})$ is also Sasakian. In particular, we have the following:
\begin{equation}
a =  \frac{{\bf{I}}(X)}{\ell(n+1)} \Longrightarrow g_{a} = g_{SE} \Longrightarrow {\rm{Ric}}(g_{\mathcal{C},a}) = 0,
\end{equation}
i.e.,  $a =  \frac{{\bf{I}}(X)}{\ell(n+1)} \Longrightarrow (\mathcal{C}(Q({\bf{L}})), g_{\mathcal{C},a}:= r^{2}g_{a} + {\rm{d}}r\otimes {\rm{d}}r,\mathcal{J}_{a})$ is Calabi-Yau. In this case, from Theorem \ref{mainproof}, we obtain a Levi-Civita Ricci-flat Hermitian manifold $(\Sigma_{\phi,c}(Q({\bf{L}})),\Omega,J_{a})$, here the complex structure $J_{a}$ is obtained from $\mathcal{J}_{a}$, for $a =  \frac{{\bf{I}}(X)}{\ell(n+1)}$. In particular, in this case we have the underlying Lee form given by $\theta \myeq -{\rm{d}}\log(r^{2})$.
\end{remark}

\begin{remark}
By following the results given in \cite{Correa} on principal elliptic bundles over flag varieties, and the construction provided by Example \ref{fundexample} (or by Remark \ref{remarkdeformcpx}), one can describe a huge class of explicit examples which illustrate the results of Theorem \ref{mainproof} and Theorem \ref{T2proof} by using tools of Lie theory. 
\end{remark}

\begin{example} Given a $(n+1)$-tuple ${\bf{a}} = (a_{0},\ldots,a_{n})$, such that $a_{j} \in \mathbb{Z}_{>0}$, $a_{j} > 1$, and $n \geq 3$, consider the Brieskorn-Pham singularity (e.g. \cite{Brieskorn1}, \cite{Pham}) given by
\begin{equation}
Y({\bf{a}}) := \Big (  z_{0}^{a_{0}} + \cdots + z_{n}^{a_{n}} = 0 \Big) \subset \mathbbm{C}^{n+1}.
\end{equation}
Let us assume that $a_{0} \leq \cdots \leq a_{n}$. Set $C = {\rm{lcm}}(a_{j}: j = 0,\ldots,n)$ and consider the $\mathbbm{C}^{\times}$-action on $Y({\bf{a}})$ defined by
\begin{equation}
\label{waction}
\lambda \cdot (z_{0},\ldots,z_{n}) = (\lambda^{\frac{C}{a_{0}}}z_{0},\ldots, \lambda^{\frac{C}{a_{n}}}z_{n}), \ \ \forall \lambda \in \mathbbm{C}^{\times}. 
\end{equation}
From above, we have the following facts \cite{Einsteinsphere}:
\begin{enumerate}
\item $X^{{\rm{orb}}}({\bf{a}}) := Y({\bf{a}})/\mathbbm{C}^{\times}$ is a K\"{a}hler orbifold;
\item $X^{{\rm{orb}}}({\bf{a}})$ is a Fano orbifold if and only if $1 < \sum_{j = 0}^{n}\frac{1}{a_{j}}$.
\end{enumerate}
Consider now the Brieskorn manifold (Figure \ref{link})
\begin{equation}
\label{Brieskornlink}
{\rm{L}}({\bf{a}}) := Y({\bf{a}}) \cap S^{2n+1}.
\end{equation}
The Brieskorn manifold ${\rm{L}}({\bf{a}})$ is a smooth $(2n-1)$-dimensional compact manifold. Moreover, by following \cite{Abe}, \cite{Takao}, and \cite{BoyerGalickifive}, we have a well-known quasi-regular Sasakian structure on ${\rm{L}}({\bf{a}})$ naturally obtained from a weighted Sasakian structure of $S^{2n+1}$. 
\begin{figure}[H]
\centering
\includegraphics[scale = .35]{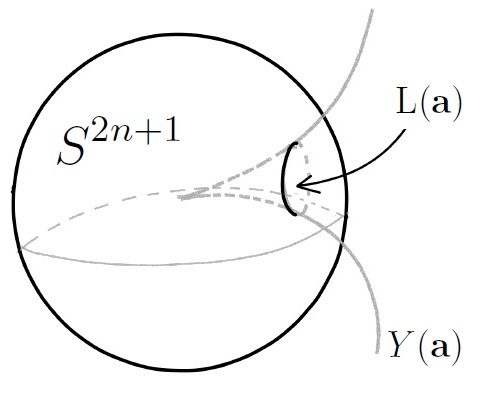}
\caption{Link of isolated singularity which represents a Brieskorn manifold ${\rm{L}}({\bf{a}})$.}
\label{link}
\end{figure}
In this setting, we have the following commutative diagram: 
\begin{center}
\begin{tikzcd}
 {\rm{L}}({\bf{a}})\arrow[d] \arrow[r,hook] & S^{2n+1} \arrow[d] \\
 X^{{\rm{orb}}}({\bf{a}}) \arrow[r,hook] & \mathbbm{P}({\bf{w}})
\end{tikzcd}
\end{center}
where $\mathbbm{P}({\bf{w}}) = (\mathbbm{C}^{n+1} - \{0\})/\mathbbm{C}^{\times}$ is the weighted projective space defined by the action described in Eq. \ref{waction}, here we denote ${\bf{w}} = (\frac{C}{a_{0}},\ldots,\frac{C}{a_{n}})$. In the above diagram the horizontal arrows are Sasakian and K\"{a}hlerian embeddings, respectively, and the vertical arrows are principal $S^{1}$ V-bundles and orbifold Riemannian submersions. Moreover, ${\rm{L}}({\bf{a}})$ is the total space of principal $S^{1}$ V-bundle over the orbifold $X^{{\rm{orb}}}({\bf{a}})$ whose the first Chern class in $H^{2}_{orb}(X^{{\rm{orb}}}({\bf{a}}),\mathbbm{Z})$ is $\frac{1}{I}c_{1}(X^{{\rm{orb}}}({\bf{a}}))$, where $I$ is the Fano index of $X^{{\rm{orb}}}({\bf{a}})$, see for instance \cite{BoyerGalickifive}. Thus, in a similar way as in the previous example, we have that ${\rm{L}}({\bf{a}})$ admits a Sasaki-Einstein metric if and only if the orbifold $X^{{\rm{orb}}}({\bf{a}})$ admits a K\"{a}hler-Einstein orbifold metric of scalar curvature $4n(n+1)$. From the Lichnerowicz obstruction \cite[Eq. (3.23)]{Lobistruction} and the recent result provided in \cite{LiuSanoTasin}, see also \cite[Conjecture 4]{Einsteinsphere}, it follows that ${\rm{L}}({\bf{a}})$ admits a Sasaki-Einstein metric if 
\begin{equation}
1 < \sum_{j = 0}^{n}\frac{1}{a_{j}} < 1 + \frac{n}{a_{n}}.
\end{equation}
In this case, we have from Theorem \ref{mainproof} (or Corollary \ref{LCRFBP}) that ${\rm{L}}({\bf{a}}) \times S^{1}$ admits a Levi-Civita Ricci-Flat Hermitian metric. Also, from Theorem \ref{T2proof}, we obtain a huge class of new examples of complex manifolds with nonnegative first Chern class that admit constant strictly negative Riemannian scalar curvature.
\end{example}

\begin{example}
\label{7spheresexample}
The construction presented in the last example above plays an important role in the study homotopy spheres and exotic spheres. Being more precise, in \cite{Hirzebruch} Hirzebruch showed that links as in Eq. (\ref{Brieskornlink}) can sometimes be homeomorphic, but not diffeomorphic to standard spheres. As shown by Egbert V. Brieskorn in \cite{Brieskorn2}, links of the form
\begin{equation}
\label{exotic7}
{\rm{L}}(2,2,2,3,6k-1) = \Big ( z_{0}^{2} + z_{1}^{2} + z_{2}^{2} + z_{3}^{3} + z_{4}^{6k-1} = 0 \Big) \cap S^{9},
\end{equation}
with $k = 1, \ldots,28$, realize explicitly all the distinct smooth structures on the $7$-sphere $S^{7}$ classified by Kervaire and Milnor in \cite{Kervaire}, see also \cite{Milnor}. Denoting by ${\bf{\Sigma}}^{7}$ any one of the 28 homotopy $7$-sphere, it follows from Theorem \ref{mainproof} (or Corollary \ref{exoticsphereLCRF}) that ${\bf{\Sigma}}^{7} \times S^{1}$ admits a Levi-Civita Ricci-Flat Hermitian metric. Moreover, from Theorem \ref{T2proof}, we have a family of Hermitian metrics $g_{\zeta}$ on ${\bf{\Sigma}}^{7} \times S^{1}$, such that $\zeta  > -1$ (see Eq. \ref{Hermitianfamily}). In this case, it follows from Eq. (\ref{scalarfunction}) that
\begin{equation}
{\rm{scal}}(g_{\zeta}) = \frac{12}{(\zeta + 1)^{2}}\bigg ( \zeta + \frac{7}{8}\bigg), \ \ \ \ \ \ \forall \zeta >  -1.
\end{equation}
Thus, we have ${\rm{scal}}(g_{\zeta}) \in (-\infty,24]$. In particular, we obtain the following:
\begin{enumerate}
\item $\zeta > -\frac{7}{8} \Longrightarrow ({\bf{\Sigma}}^{7} \times S^{1},g_{\zeta},J)$ has constant strictly positive Riemannian scalar curvature;
\item $\zeta = -\frac{7}{8}\Longrightarrow ({\bf{\Sigma}}^{7} \times S^{1},g_{\zeta},J)$ has constant zero Riemannian scalar curvature;
\item $-1 < \zeta < -\frac{7}{8} \Longrightarrow ({\bf{\Sigma}}^{7} \times S^{1},g_{\zeta},J)$ has constant strictly negative Riemannian scalar curvature.
\end{enumerate}
Notice that ${\rm{scal}}(g_{\zeta}) \to -\infty$ as $\zeta \to -1$. According to  \cite{GHM}, the existence of an exotic $8$-sphere ${\bf{\Sigma}}^{8}$ implies the existence of 30 different differentiable structures on $S^{7} \times S^{1}$, i.e., besides those 28 obtained from ${\bf{\Sigma}}^{7} \times S^{1}$, we also have an additional one obtained from $({\bf{\Sigma}}^{7} \times S^{1}) \# {\bf{\Sigma}}^{8}$, where ${\bf{\Sigma}}^{7}$ is any one of the 28 homotopy $7$-spheres. It would be interesting to know whether the above results also hold for $({\bf{\Sigma}}^{7} \times S^{1}) \# {\bf{\Sigma}}^{8}$.
\end{example}

\begin{remark}
Another source of examples which illustrate the results of Theorem \ref{mainproof} and Theorem \ref{T2proof} is provided by manifolds of the form $\mathcal{M}_{k} \times S^{1}$, such that $\mathcal{M}_{k} = k\#(S^{2}\times S^{3})$, $k \geq 1$. Actually, it is known that $\mathcal{M}_{k}$, $k \geq 1$, admits infinitely many Sasaki-Einstein structures, see for instance \cite{BoyerGalicki}, \cite{Coevering}. In particular, from Theorem \ref{mainproof} one can always solve the equation $\mathfrak{R}^{(1)}(\Omega) = 0$ on $\mathcal{M}_{k} \times S^{1}$, $k \geq 1$.
\end{remark}

\appendix

\section{Remarks on Chern-Ricci flow}
\label{Appendix}
On a Hermitian manifold $(M,\Omega_{0},J)$, a solution of the Chern-Ricci flow starting at $\Omega_{0}$ is given by a smooth family of Hermitian metrics $\Omega = \Omega(t)$ satisfying 
\begin{equation}
\begin{cases}
\displaystyle \frac{\partial}{\partial t} \Omega = -{\rm{Ric}}^{(1)}(\Omega), \ \ 0 \leq t < T, \\
\Omega(0) = \Omega_{0},
\end{cases}
\end{equation}
for $T \in (0,\infty]$, see for instance \cite{Gill} and \cite{TosattiWeinkove}. Inspired by the results on Hopf manifolds provided in \cite{TosattiWeinkove}, \cite{GillSmith}, and \cite{TosattiWeinkove2}, we observe that the ideas involved in the proof of Theorem \ref{mainproof} can be used to obtain explicit solutions of the Chern-Ricci flow on compact Hermitian Weyl-Einstein manifolds. More precisely, we have the following result:

\begin{theorem}
\label{T3proof}
Let $(M,g,J)$ be a compact Hermitian Weyl-Einstein manifold, then there exists an explicit solution $g(t)$ of the Chern-Ricci flow on $M$ for $ t \in [0,\frac{2}{n})$, starting at $g$, satisfying the following properties:
\begin{enumerate}
\item ${\rm{Vol}}(M,g(t)) \to 0$ as $t \to \frac{2}{n}$ (i.e. $g(t)$ is finite-time collapsing);
\item $\lim_{t \to \frac{2}{n}}g(t) = h_{T}$, where $h_{T}$ is a nonnegative symmetric tensor on $M$;
\item The Chern scalar curvature of $g(t)$ blows up like $(n-1)/(\frac{2}{n}- t)$;
\item ${\rm{scal}}(g(t)) \to -\infty$ as $t \to \frac{2}{n}$;
\item $\lim_{t \to \frac{2}{n}}d_{GH}\big ((M,d_{t}),(S^{1},d_{S^{1}}) \big ) = 0$,
\end{enumerate}
where $d_{t}$ is the distance induced by $g(t)$ on $M$ and $d_{S^{1}}$ is the distance on the unit circle $S^{1}$ induced by a suitable scalar multiple of the standard Riemannian metric.
\end{theorem}

\begin{remark}
In order to prove Theorem \ref{T3proof}, we proceed highlighting the background on Sasaki-Einstein geometry underlying all our main results. In fact, in what follows we shall prove the results of Theorem \ref{T3proof} for compact Hermitian manifolds as in the proof of Theorem \ref{mainproof}, i.e., for compact Hermitian manifolds of the form $(\Sigma_{\phi,c}(Q),g,J)$, such that $g$  is a Hermitian Weyl-Einstein metric induced by the Calabi-Yau structure of $\mathcal{C}(Q)$. We also include in the proof some comments relating our results with previous known results on Hopf manifolds. 
\end{remark}

Before we start the proof, we recall that the Gromov-Hausdorff distance of metric spaces can be defined as follows (see for instance \cite{Burago}). Given two metric spaces $(X,d_{X})$ and $(Y,d_{Y})$, a {\textit{correspondence}} between the underlying sets $X$ and $Y$ is a subset $R \subseteq  X \times Y$ satisfying the following property: for every $x \in X$ there exists at leas one $y \in Y$, such that $(x,y) \in X\times Y$, and similarly for every $y \in Y$ there exists an $x \in X$, such that $(x,y) \in X \times Y$. Let us denote by $\mathcal{R}(X,Y)$ the set of all correspondences between $X$ and $Y$. Now we consider the following definition
\begin{definition}
Let $R \in \mathcal{R}(X,Y)$ be a correspondence between two metric spaces $(X,d_{X})$ and $(Y,d_{Y})$. The \textit{distortion} of $R$ is defined by 
\begin{equation}
{\rm{dis}}(R) = \sup \Big  \{ |d_{X}(x,x') - d_{Y}(y,y')| \ \ \big | \ \ (x,y), (x',y') \in R \ \Big\}.
\end{equation}
\end{definition}

From above, we can define the Gromov-Hausdorff distance between two metric spaces as follows.

\begin{definition}
We define the Gromov-Hausdorff distance of any two metric spaces $(X,d_{X})$ and $(Y,d_{Y})$ as being 
\begin{equation}
\label{GHdistance}
d_{GH}\big ((X,d_{X}),(Y,d_{Y}) \big ) = \frac{1}{2}\inf \Big  \{ {\rm{dis}}(R) \ \ \Big | \ \ R  \in \mathcal{R}(X,Y) \Big\}.
\end{equation}
\end{definition}

Now we can prove Theorem \ref{T3proof}.

\begin{proof}(Theorem \ref{T3proof}) As we have seen, given a compact Hermitian Weyl-Einstein  manifold of the form $(\Sigma_{\phi,c}(Q),g,J)$, it follows that 
\begin{equation}
\Omega = -{\rm{d}}(J\theta) + \theta \wedge J\theta \ \ \ {\text{and}} \ \ \ {\rm{Ric}}^{(1)}(\Omega) = - \frac{n}{2}{\rm{d}}(J\theta).
\end{equation}
Therefore, we set $\Omega_{0} := \Omega$ and 
\begin{equation}
\label{flowsolution}
\Omega(t) := \Omega_{0} - t{\rm{Ric}}^{(1)}(\Omega_{0}) = -\Big (1-\frac{n}{2}t \Big){\rm{d}}(J\theta) + \theta \wedge J\theta.
\end{equation}
Since $-{\rm{d}}(J\theta) \geq 0$, we have that $\Omega(t)$ is a Hermitian metric for all $t \in [0,\frac{2}{n})$. Moreover, a straightforward computation shows that 
\begin{equation}
\label{volumeflow}
\Omega(t)^{n} = \Big (1-\frac{n}{2}t \Big)^{n-1}\Omega_{0}^{n}.
\end{equation}
Thus, we have ${\rm{Ric}}^{(1)}(\Omega(t)) = {\rm{Ric}}^{(1)}(\Omega_{0})$,  $\forall t \in [0,\frac{2}{n})$. From this, we conclude that 
\begin{equation}
\frac{\partial}{\partial t} \Omega(t) = - {\rm{Ric}}^{(1)}(\Omega_{0}) = - {\rm{Ric}}^{(1)}(\Omega(t)), 
\end{equation}
for all $t \in [0,\frac{2}{n})$. Therefore, the family of Hermitian metrics $\Omega(t) = \Omega_{0} - t{\rm{Ric}}^{(1)}(\Omega_{0})$ on $\Sigma_{\phi,c}(Q)$ gives a solution of the Chern-Ricci flow on the maximal existence interval $[0,T)$, such that $T = \frac{2}{n}$. Let us observe that 
\begin{equation}
{\rm{d}}\Omega(t) = \theta(t) \wedge \Omega(t), \ \ {\text{such that }} \ \ \theta(t) = \frac{T}{T-t}\theta,
\end{equation}
for all $t \in [0,T)$. Moreover, one can easily verify that $\theta(t)$ is parallel with respect to the Levi-Civita connection induced by $\Omega(t)$. Therefore, we have that $\Omega(t)$ is Vaisman for all $t \in [0,T)$. Also, from Eq. (\ref{volumeflow}) and Eq. (\ref{flowsolution}), it follows that
\begin{equation}
\lim_{t \to T}{\rm{Vol}}\big (\Sigma_{\phi,c}(Q), \Omega(t)\big ) = 0 \ \ \ {\text{and}} \ \ \ \lim_{t \to T}\Omega(t) = \theta \wedge J\theta, 
\end{equation}
i.e., $\Omega(t)$ is finite-time collapsing \cite{TosattiWeinkove2}. Observing that $\Omega_{T}:= \theta \wedge J\theta$ is a nonnegative $(1,1)$-form, we conclude that the behavior of the Chern-Ricci flow $\Omega(t)$ is quite similar to the behavior of the Chern-Ricci flow on Hopf manifolds provided in \cite[Porposition 1.8]{TosattiWeinkove}. Further, as it was shown in \cite{GillSmith}, finite-time singularities are characterized by the blow-up of the Chern scalar curvature. This last fact can be easily verified for the Chern-Ricci flow describe in Eq. (\ref{flowsolution}). In fact, by a straightforward computation one can show that
\begin{equation}
{\rm{d}}(J \theta) \wedge \Omega(t)^{n-1} = -(n-1)!\bigg ( \frac{n-1}{1-\frac{n}{2}t}\bigg) \frac{\Omega(t)^{n}}{n!}.
\end{equation}
Thus, since ${\rm{Ric}}^{(1)}(\Omega(t)) = {\rm{Ric}}^{(1)}(\Omega_{0})$,  $\forall t \in [0,T)$, we obtain the following
\begin{equation}
\frac{s_{C}(\Omega(t))}{n} \Omega(t)^{n} = -\frac{n}{2}{\rm{d}}(J \theta) \wedge \Omega(t)^{n-1} = \frac{1}{2}\bigg ( \frac{n-1}{1-\frac{n}{2}t}\bigg) \Omega(t)^{n}.
\end{equation}
From above we have
\begin{equation}
\label{Chernscalarflow}
s_{C}(\Omega(t)) =  \frac{n}{2}\bigg ( \frac{n-1}{1-\frac{n}{2}t} \bigg ) = \frac{n-1}{T - t}, 
\end{equation}
for all $0 \leq t < T$. It follows that $s_{C}(\Omega(t)) \to +\infty$ as $t \to T$ (cf. \cite{GillSmith}). Also, denoting by $g(t)$ the underlying Riemannian metric associated to $\Omega(t)$, it follows from Eq. (\ref{scalarfunction}) that
\begin{equation}
\label{Rscalarflow}
{\rm{scal}}(g(t)) = \frac{n(n-1)}{(1-\frac{n}{2}t)^{2}} \bigg [\frac{2n-1}{2n} - \frac{n}{2}t\bigg ] = \frac{n(n-1)}{(T-t)^{2}} \bigg [T - \frac{1}{n^{2}} - t\bigg ],
\end{equation}
for all $0 \leq t < T$. Hence, we obtain ${\rm{scal}}(g(t))  \to -\infty$ as $t \to T$. In particular, we have:
\begin{enumerate}
\item[(a)] $0 \leq t < T-\frac{1}{n^{2}} \Longrightarrow {\rm{scal}}(g(t)) > 0$;
\item[(b)] $t = T-\frac{1}{n^{2}} \Longrightarrow {\rm{scal}}(g(t)) = 0$;
\item[(c)] $T-\frac{1}{n^{2}} < t < T  \Longrightarrow {\rm{scal}}(g(t)) < 0$.
\end{enumerate}
From above, we obtain a complete picture of the behavior of $ {\rm{scal}}(g(t))$, $t \in [0,T)$. Further, by means of a suitable change in the argument presented in \cite[$\S$ 4]{TosattiWeinkove2} one can show that 
\begin{equation}
\big (\Sigma_{\phi,c}(Q),d_{t}\big) \xrightarrow{\text{G.H.}} \big (S^{1},d_{S^{1}}\big), \ \ {\text{as}} \ \ t \to T, 
\end{equation}
where $d_{t}$ is the distance induced by $g(t)$ and $d_{S^{1}}$ is the distance on the unit circle $S^{1}$ induced by a suitable scalar multiple of the standard Riemannian metric. 
From Definition \ref{GHdistance}, we say that $\big (\Sigma_{\phi,c}(Q),d_{t}\big) \xrightarrow{\text{G.H.}} \big (S^{1},d_{S^{1}}\big)$, as $t \to T$, if 
\begin{equation}
\label{GHconv}
\lim_{t\to T}d_{GH}\big ((\Sigma_{\phi,c}(Q),d_{t}),(S^{1},d_{S^{1}}) \big ) =0.
\end{equation}
In order to conclude the proof, consider $F \colon Q \times \mathbbm{R} \to \mathbbm{R}$, such that $F(x,\varphi) = -2\varphi$, $\forall (x,\varphi) \in Q \times \mathbbm{R}$. From this, we set $F_{c} \colon \Sigma_{\phi,c}(Q) \to S^{1}$, such that
\begin{equation}
F_{c}([x,\varphi]) := \exp \bigg ( \frac{2\pi\sqrt{-1}F(x,\varphi)}{\log(c^{2})}\bigg), 
\end{equation}
for all $[x,\varphi] \in \Sigma_{\phi,c}(Q)$. Since 
\begin{equation}
F(\phi(x),\varphi + n \log(c)) = F(x,\varphi) - n\log(c^{2}),
\end{equation}
it follows that $F_{c}$ is a well-defined smooth map. Considering the canonical angular $1$-form ${\rm{d}}\sigma$ on $S^{1}$, a straightforward computation shows us that
\begin{equation}
F_{c}^{\ast}({\rm{d}}\sigma) = \frac{\theta}{\log(c^{2})}.
\end{equation}
From above, since $\theta$ is a non-vanishing $1$-form, it follows that $F_{c}$ is a submersion. Also, we notice that $\ker((F_{c})_{\ast}) = \ker(\theta(t))$, $\forall 0 \leq t < T$. Since
\begin{equation}
\label{decompositionmetric}
g(t) = \underbrace{-\Big (1-\frac{n}{2}t \Big){\rm{d}}(J\theta)(\mathbbm{1} \otimes J)}_{h(t)} + \underbrace{\theta \otimes \theta + J\theta \otimes J\theta}_{h_{T}},
\end{equation}
see Eq. (\ref{flowsolution}), by considering the $g(t)$-orthogonal complement $\ker((F_{c})_{\ast})^{\perp}$, it follows that
\begin{equation}
g(t)|_{\ker((F_{c})_{\ast})^{\perp}} = \theta \otimes \theta = \lambda^{2} F_{c}^{\ast}\big ({\rm{d}}\sigma \otimes {\rm{d}}\sigma\big),
\end{equation}
where $\lambda = \log(c^{2})$, i.e., $F_{c} \colon (\Sigma_{\phi,c}(Q),g(t)) \to (S^{1},\lambda^{2}{\rm{d}}\sigma \otimes {\rm{d}}\sigma)$ is a Riemannian submersion. Since $\Sigma_{\phi,c}(Q)$ is compact, we have that $F_{c} \colon \Sigma_{\phi,c}(Q) \to S^{1}$ is in fact a locally trivial fiber bundle with typical fiber diffeomorphic to $Q$ (e.g. \cite{Besse}). From $h_{T}$ given in Eq. (\ref{decompositionmetric}), we set
\begin{equation}
\mathcal{D}:= \big \{X \in T\Sigma_{\phi,c}(Q) \ \big | \ h_{T}(X,Y) = 0, \ \forall Y \big \}.
\end{equation}
Let us denote by $S \subset \Sigma_{\phi,c}(Q)$ a generic fiber of $F_{c}$. Denote also by $\mathcal{D}_{S}$ the distribution $\mathcal{D}$ restricted to $S$. By construction, since $\ker((F_{c})_{\ast}) = \ker(\theta)$, we have that $(S,\eta_{S}:=i^{\ast}(J\theta))$ is a contact manifold, where $i \colon S \hookrightarrow \Sigma_{\phi,c}(Q)$ is the natural inclusion, see for instance \cite{Dragomir}. Hence, we obtain the following description
\begin{equation}
\mathcal{D}_{S}= \big \{X \in TS \ \big | \ \eta_{S}(X) = 0 \big \},
\end{equation}
i.e. $\mathcal{D}_{S}$ is the contact distribution on $S$ induced by $\eta_{S}$. Since every contact distribution is a bracket-generating distribution, it follows from Chow's theorem \cite{Montgomery} that any two points of $S$ can be connected by a smooth path tangent to $\mathcal{D}_{S}$. It is worth observing that $g(t)|_{\mathcal{D}_{S}} = h(t)|_{\mathcal{D}_{S}}$, where $h(t)$ is given as in Eq. (\ref{decompositionmetric}). Given $p,q \in \Sigma_{\phi,c}(Q)$, let $\alpha$ be a path connecting $F_{c}(p)$ and $F_{c}(q)$ in $S^{1}$. Since $(\Sigma_{\phi,c}(Q),g(t))$ is complete, there exists a lift $\widetilde{\alpha}$ of $\alpha$ starting at $p$ and tangent to $\ker((F_{c})_{\ast})^{\perp}$. Let us denote by $z \in \Sigma_{\phi,c}(Q)$ the endpoint of the path $\widetilde{\alpha}$ (Figure \ref{paths}). From this, we have $d_{t}(p,z) \leq d_{S^{1}}(F_{c}(p),F_{c}(z)) = d_{S^{1}}(F_{c}(p),F_{c}(q))$, where $d_{S^{1}}$ denotes the distance induced by $\lambda^{2}{\rm{d}}\sigma \otimes {\rm{d}}\sigma$ on $S^{1}$. Hence, we obtain
\begin{equation}
d_{t}(p,q) \leq d_{t}(p,z) + d_{t}(z,q) \leq d_{S^{1}}(F_{c}(p),F_{c}(q)) +  d_{t}(z,q).
\end{equation}
Since $z,q \in F_{c}^{-1}(F_{c}(q)) = S$, we have a smooth path $\gamma \colon [0,1] \to \Sigma_{\phi,c}(Q)$ tangent to $\mathcal{D}_{S}$ connecting $z$ and $q$ (Figure \ref{paths}). 

\begin{figure}[H]
\includegraphics[scale = .22]{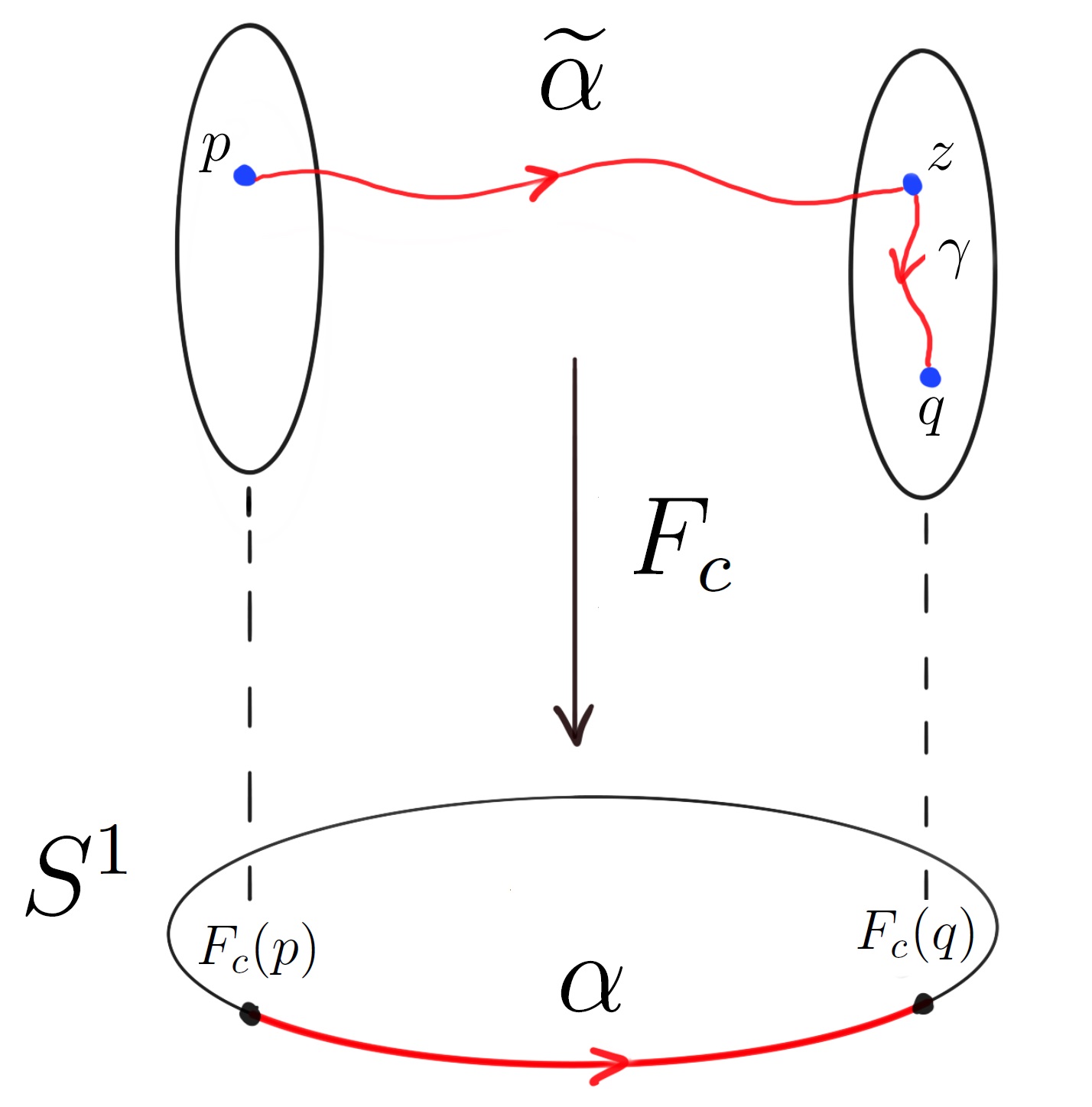}
\caption{Representation of horizontal and vertical paths used to estimate the distance between $p$ and $q$.}
\label{paths}
\end{figure}

Therefore, since $g(t)|_{\mathcal{D}_{S}} = h(t)|_{\mathcal{D}_{S}}$ and $\Sigma_{\phi,c}(Q)$ is compact, we obtain
\begin{equation}
 d_{t}(z,q) \leq \int_{0}^{1}||\gamma'(s)||_{g(t)}ds = \Big(1-\frac{n}{2}t \Big)\int_{0}^{1}||\gamma'(s)||_{g(0)}ds\leq C \Big (1-\frac{n}{2}t \Big),
\end{equation}
where $C := {\rm{diam}}\big(\Sigma_{\phi,c}(Q),g(0)\big)$. Hence, it follows that 
\begin{equation}
\label{estimate}
d_{t}(p,q) - d_{S^{1}}(F_{c}(p),F_{c}(q)) \leq  C \Big (1-\frac{n}{2}t \Big),
\end{equation}
for all $t \in [0,T)$ and for all $p,q \in \Sigma_{\phi,c}(Q)$. From $F_{c} \colon \Sigma_{\phi,c}(Q) \to S^{1}$, we set
\begin{equation}
R(F_{c}) := \big \{ (p,F_{c}(p)) \in \Sigma_{\phi,c}(Q) \times S^{1} \ \big  | \  p \in \Sigma_{\phi,c}(Q) \big \}.
\end{equation}
Since $F_{c}$ is a surjective map, it follows that $R(F_{c}) \in \mathcal{R}(\Sigma_{\phi,c}(Q),S^{1})$. From Eq. (\ref{estimate}), we conclude that 
\begin{equation}
d_{GH}\big ((\Sigma_{\phi,c}(Q),d_{t}),(S^{1},d_{S^{1}}) \big ) \leq \frac{1}{2}{\rm{dis}}(R(F_{c})) < \frac{nC}{2} \Big (T-t \Big).
\end{equation}
Therefore, it follows that 
\begin{equation}
\lim_{t \to T}d_{GH}\big ((\Sigma_{\phi,c}(Q),d_{t}),(S^{1},d_{S^{1}}) \big ) = 0.
\end{equation}
As it can be seen, the arguments provided above generalize certain ideas introduced in \cite[\S 4]{TosattiWeinkove2} for Hopf manifolds. Following Theorem \ref{isovaisman} and the above constructions, we conclude the proof of Theorem \ref{T3proof}. \end{proof}

\begin{example}
Consider ${\bf{\Sigma}}^{7} \times S^{1}$, where ${\bf{\Sigma}}^{7}$ is any one of the 28 homotopy $7$-spheres. Fixed a Hermitian Weyl-Einstein metric $\Omega_{WE}$ on ${\bf{\Sigma}}^{7} \times S^{1}$, it follows from Theorem \ref{T3proof} that there exists a family of Hermitian metrics $\Omega(t)$, $t \in [0,\frac{1}{2})$, satisfying 
\begin{equation}
\begin{cases}
\displaystyle \frac{\partial}{\partial t} \Omega = -{\rm{Ric}}^{(1)}(\Omega), \ \ 0 \leq t < \frac{1}{2}, \\
\Omega(0) = \Omega_{WE}.
\end{cases}
\end{equation}
Moreover, it follows form Eq. \ref{Chernscalarflow} and from Eq. \ref{Rscalarflow} that 
\begin{equation}
s_{C}(\Omega(t)) = \frac{6}{1 - 2t}, \ \ \ \ {\text{and}} \ \ \ \ {\rm{scal}}(g(t)) = \frac{48}{(1 - 2t)^{2}} \bigg [ \frac{7}{16} - t \bigg].
\end{equation}
From above, we obtain the following:
\begin{enumerate}
\item[(a)] $0 \leq t < \frac{7}{16} \Longrightarrow {\rm{scal}}(g(t)) > 0$;
\item[(b)] $t = \frac{7}{16} \Longrightarrow {\rm{scal}}(g(t)) = 0$;
\item[(c)] $\frac{7}{16} < t < \frac{1}{2}  \Longrightarrow {\rm{scal}}(g(t)) < 0$.
\end{enumerate}
Regarding ${\bf{\Sigma}}^{7} \times S^{1}$ as a suspension by $({\rm{id}},\sqrt{{\rm{e}}})$ of ${\bf{\Sigma}}^{7}$, from Theorem \ref{T3proof} we conclude that
\begin{equation}
\lim_{t \to \frac{1}{2}}d_{GH}\big (({\bf{\Sigma}}^{7} \times S^{1},d_{t}),(S^{1},d_{S^{1}}) \big ) = 0,
\end{equation}
where $d_{t}$ is the distance induced by $g(t)$ on ${\bf{\Sigma}}^{7} \times S^{1}$ and $d_{S^{1}}$ is the distance on the unit circle $S^{1}$ induced by the standard Riemannian metric.
\end{example}

\subsection*{Conflict of interest statement} The author declares that there is no conflict of interest.

\end{document}